\documentclass[a4paper,11pt]{amsart}

\usepackage{etoolbox}
\usepackage[utf8]{inputenc}    
\usepackage[T1]{fontenc}
\usepackage{amssymb,amsthm,amsmath,verbatim}
\usepackage{t1enc}
\usepackage{enumerate}
\usepackage{bbm}
\usepackage{color}
\usepackage{tikz}
\usepackage{breqn}
\usepackage{xmpmulti}
\usepackage{psfrag}
\usepackage[usenames,dvipsnames]{pstricks}
\usepackage{epsfig}
\usepackage{pst-grad}
\usepackage{pstricks,pst-node}
\usepackage{float}
\usepackage{todonotes}
\usepackage[left=3.5 cm, right=3.5cm]{geometry}
\usepackage[backend=biber,style=alphabetic,natbib=true,citereset=subsection,maxbibnames=9, maxcitenames=9, autolang=other, bibencoding=auto]{biblatex}
\numberwithin{equation}{section}

\newenvironment{tproof}{
  
 \begin{proof}
}{\end{proof}}

\newenvironment{cproof}{
  
  \begin{proof}
}{\end{proof}}

\newcommand{\force}{{\hspace{0.02 cm}\Vdash}}

\newtheorem{prop}{Proposition}[section]

\newtheorem{lemma}[prop]{Lemma}
\newtheorem{fact}[prop]{Fact}
\newtheorem{cor}[prop]{Corollary}
\newtheorem{thm}[prop]{Theorem}

\newtheorem{clm}[prop]{Claim}
\newtheorem{obs}[prop]{Observation}

\newtheorem{quest}[prop]{Question}

\theoremstyle{definition}
\newtheorem{dfn}[prop]{Definition}

\newcommand{\mc}[1]{\mathcal{#1}}
\newcommand{\mb}[1]{\mathbb{#1}}

\newcommand{\oo}{\omega}
\newcommand{\uhr}{\upharpoonright}
\newcommand{\omg}{{\omega_1}}

\DeclareMathOperator{\MA}{MA}

\DeclareMathOperator{\cf}{cf}

\DeclareMathOperator{\cl}{cl}

\DeclareMathOperator{\trcl}{trcl}
\DeclareMathOperator{\lth}{lth}

\def\<{\left\langle}
\def\>{\right\rangle}
\def\br#1;#2;{\bigl[ {#1} \bigr]^ {#2} }

\newcommand{\mf}[1]{\mathfrak{#1}}

\newcommand{\set}[1]{\textmd{set}({#1})}

\newcommand{\setm}{\setminus}

\newcommand{\subs}{\subset}
\newcommand{\dom}{\operatorname{dom}}

\newcommand{\scc}[1]{\textmd{succ}_{#1}}

\title{Towers and gaps at uncountable cardinals}
\date{\today}

  \author[V. Fischer]{Vera Fischer}
 \address[V. Fischer]{Universit\"at Wien,
Kurt G\"odel Research Center for Mathematical Logic, Wien, Austria}
 \email{vera.fischer@univie.ac.at}
  \urladdr{http://www.logic.univie.ac.at/$\sim $vfischer/}

   \author[D. C. Montoya]{Diana Carolina Montoya}
     \address[D. C. Montoya]{Universit\"at Wien,
Kurt G\"odel Research Center for Mathematical Logic, Wien, Austria}
   \email{dcmontoyaa@gmail.com}
  \urladdr{http://www.logic.univie.ac.at/$\sim $montoyd8/}
   
   \author[J. Schilhan]{Jonathan Schilhan}
  \address[J. Schilhan]{Universit\"at Wien,
Kurt G\"odel Research Center for Mathematical Logic, Wien, Austria}
 \email{jonathan.schilhan@univie.ac.at}
   
  \author[D.T. Soukup]{D\'aniel T. Soukup}
  \address[D.T. Soukup]{Universit\"at Wien,
Kurt G\"odel Research Center for Mathematical Logic, Wien, Austria}
 \email[Corresponding author]{daniel.soukup@univie.ac.at}
 \urladdr{http://www.logic.univie.ac.at/$\sim  $soukupd73/}

\makeatletter
\newtheorem*{rep@theorem}{\rep@title}
\newcommand{\newreptheorem}[2]{%
\newenvironment{rep#1}[1]{%
 \def\rep@title{#2 \ref{##1}}%
 \begin{rep@theorem}}%
 {\end{rep@theorem}}}
\makeatother

\newreptheorem{corollary}{Corollary}
\newreptheorem{theorem}{Theorem}


\begin{document}
 \begin{abstract}  
Our goal is to study the pseudo-intersection and tower numbers on uncountable regular cardinals, whether these two cardinal characteristics are necessarily equal, and related problems on the existence of gaps. First, we prove that either $\mf p(\kappa)=\mf t(\kappa)$ or there is a $(\mf p(\kappa),\lambda)$-gap of club-supported slaloms for some $\lambda< \mf p(\kappa)$. While the existence of such gaps is unclear, this is a promising step to lift Malliaris and Shelah's proof of $\mf p=\mf t$ to uncountable cardinals. We do analyze gaps of slaloms and, in particular, show that $\mf p(\kappa)$ is always regular; the latter extends results of Garti. Finally, we turn to club variants of  $\mf p(\kappa)$ and present a new model for the inequality $\mathfrak{p}(\kappa) = \kappa^+ <  \mathfrak{p}_{\cl}(\kappa) = 2^\kappa$. In contrast to earlier arguments by Shelah and Spasojevic,  we achieve this by adding $\kappa$-Cohen reals and then successively diagonalising the club-filter; the latter is shown to preserve a Cohen witness to $\mathfrak{p}(\kappa) = \kappa^+$.
\end{abstract}

\maketitle

\section{Introduction}

The classical tower and pseudo-intersection numbers ($\mf t$ and $\mf p$, respectively) have played a significant role in the study of cardinal characteristics of the continuum and special subsets of the reals. The cardinal $\mf t$ is the minimum size of a tower of subsets of $\omega$ i.e., a $\subseteq^\ast$-decreasing sequence of subsets of $\omega$ with no infinite pseudo-intersection, and $\mf p$ is the minimum size of a base for a filter on $\omega$ with no infinite pseudo-intersection. 

It was unknown  for a long time whether these two cardinals coincide.  Rothberger proved in \cite{Roth39} and \cite{Roth48} that $\mf p \leq \mf t$ and also that if $\mf p = \aleph_1$ then $\mf t=\aleph_1$ as well. Results from  the years after Rothberger's paper suggest that the consistency of $\mf p < \mf t$ seemed plausible to many set-theorists working in the area. Hence, the groundbreaking result of Malliaris and Shelah \cite{MaSh:CS} came with considerable surprise: the cardinals $\mf t$ and $\mf p$ are provably equal.

Meanwhile, recent years have seen an increased interest in the study of the combinatorics of the generalized Baire spaces $\kappa^\kappa$, when $\kappa$ is an uncountable regular cardinal. This fruitful new area of research provided  extensions of classical results from the $\kappa=\omega$ case often requiring the development of completely new machinery to do so. Striking new inequalities were proved as well between cardinal invariants of $\kappa^\kappa$ which are known to fail in the classical setting. 
Thus a natural question becomes: Does Malliaris and Shelah's result mentioned above lift to the uncountable?

The goal of the current paper is the study of the higher analogues of the tower and pseudo-intersection numbers. We start with some basic definitions.

\begin{dfn} Let $\kappa$ be a regular uncountable cardinal.
\begin{enumerate}
\item Let $\mathcal{F}$ be a family of subsets of $\kappa$. We say that $\mathcal{F}$ has the \emph{strong intersection property} (in short, SIP) if for any subfamily $\mathcal{F}' \subseteq \mathcal{F}$ of size $<\kappa$, the intersection $\bigcap \mc F'$ has size $\kappa$. 
\item We say that $A \subseteq \kappa$ is a \emph{pseudo-intersection} of $\mathcal{F}$ if $A \subseteq^\ast F$ for all $F \in \mathcal{F}$.\footnote{As usual, $A \subseteq^\ast F$ means that $A\setm F$ has size $<\kappa$.}
\item A \emph{tower} $\mathcal{T}$ is a $\subseteq^*$-well-ordered family of subsets of $\kappa$ with the SIP that has no pseudo-intersection of size $\kappa$. 
\end{enumerate}
\end{dfn}

In the countable case, any $\subseteq^*$-well-ordered family of infinite sets has the SIP. However, for uncountable $\kappa$, the SIP requirement is necessary as there are countable $\subseteq^*$-decreasing families of subsets of $\kappa$ with no pseudo-intersection of size $\kappa$.\footnote{E.g., consider a partition of $\kappa$ into sets $\{X_n:n<\oo\}$ and look at $\mc T=\{\bigcup_{m\geq n}X_m:n\in \oo\}$.}

\begin{dfn}[The pseudo-intersection and tower number]\label{pt}\hfill
\begin{enumerate}
\item  The \emph{pseudo-intersection number} for $\kappa$, denoted by $\mathfrak{p}(\kappa)$, is defined as the minimal size of a family $\mathcal{F}\subset [\kappa]^\kappa$ which has the SIP but no pseudo-\linebreak intersection of size $\kappa$.
\item The \emph{tower number} for $\kappa$, denoted by $\mathfrak{t}(\kappa)$, is defined as the minimal size of a tower $\mathcal{T}\subset [\kappa]^\kappa$  of subsets of $\kappa$.
\item $\mf p_{\cl}(\kappa)$ is the minimal size of a family $\mathcal{F}$ of \emph{club} subsets of $\kappa$ with no pseudo-intersection of size $\kappa$.
\item $\mf t_{\cl}(\kappa)$ the minimal size of a tower $\mathcal{T}$ of \emph{club} subsets of $\kappa$.
\end{enumerate}
\end{dfn}
 Note that
in the definition of  $\mf p_{\cl}(\kappa)$ and $\mf t_{\cl}(\kappa)$, there is no need to assume the SIP as any family of clubs has the strong intersection property. For higher analogues of $\mf p$ and $\mf t$ which do not require the SIP property see Section \ref{app:1} of our Appendix.

The study of the above cardinal invariants was initiated by Garti \cite{SG:pity} and  one of the results which motivated the work on this project is the following:

\begin{thm}\label{garti_r}\cite{SG:pity}
Let $\kappa$ be an uncountable cardinal such that $\kappa^{<\kappa}=\kappa$.
\begin{enumerate}
 \item If $\mf p(\kappa)= \kappa^+$, then $\mf t(\kappa)=\kappa^+$.
    \item If $\cf(2^\kappa) \in \{\kappa^+, \kappa^{++}\}$, then $\mf p(\kappa)= \mf t(\kappa)$.
    \item $\cf(\mf p(\kappa)) \neq \kappa$.
\end{enumerate}
\end{thm}

Related consistency results also appear in the very recent paper of Ben-Neria and Garti \cite{ben2019configurations}.

%Another result, which is of interest for our work is a theorem of Shelah and Spasojevi\'{c} %\cite{shelah2002cardinal} stating that whenever $\kappa^{<\kappa}=\kappa$, $\mf b(\kappa)\leq %\mf t(\kappa)$ and if $\kappa\leq \mu<\mf t(\kappa)$ then $2^\mu=2^\kappa$.

\bigskip
\noindent
{\emph{Structure of the paper:}} 
The current paper is structured as follows. In Section~\ref{pktk} we introduce a natural higher analogue of the notion of a gap which gives an interesting analogue of a theorem of Malliaris-Shelah, which is central to the proof of $\mf p =\mf t$. More precisely, we work with club-supported gaps of slaloms\footnote{Note that there are no real gaps of function in $\kappa^\kappa$. Indeed,  there is no infinite $<^*$-decreasing sequence of functions in $\kappa^\kappa$ when $\kappa\geq cf(\kappa)>\oo$.}  (see Definition~\ref{def.D-supported.slalom}) and prove:

\begin{thm}
Let $\kappa$ be a regular cardinal such that $\kappa^{<\kappa}=\kappa$. Either $\mf p(\kappa)=\mf t(\kappa)$ or there is a $\lambda<\mf p(\kappa)$ and club-supported $(\mf p(\kappa),\lambda)$-gap of slaloms.
\end{thm}

In Section~\ref{sec:gaps}, we study the possible sizes of gaps of slaloms which
leads in particular to the following result (see Corollary~\ref{pk_regular}):

\begin{thm}
For any uncountable, regular $\kappa$, $\mf p(\kappa)$ is regular. 
\end{thm}

Additionally, we consider a higher analogue of Martin's Axiom 
(see Definition~\ref{def.MA-good-Knaster}) and its effect on certain club-supported gaps of slaloms (see Theorem~\ref{thm.MA-good-Knaster}). 
%An analogous duality was proved in the classical case on $\omega$, however our proofs are quite different.  We also extend results of %Shelah \cite{Sh:pt} showing that under certain assumptions, there are no gaps of type...  \todo{Add details} At this point, whether %$\mf p(\kappa)=\mathfrak{t}(\kappa)$ for any uncountable $\kappa$, remains open.
In Section~\ref{pcltcl}, we look at the relation of $\mf p(\kappa)$ and its restriction to the club filter, $\mf p_{\cl}(\kappa)$. Apart from showing that $\mf p_{\cl}(\kappa)=\mf t_{\cl}(\kappa)=\mf b(\kappa)$ (see Observation~\ref{obs:pcl}), we prove:

\begin{thm}
(GCH) For any regular uncountable $\kappa < \lambda$, where $\kappa = \kappa^{<\kappa}$, there is a $\kappa$-closed, $\kappa^+$-cc forcing extension in which $\mathfrak{p}(\kappa) = \kappa^+ <  \mathfrak{p}_{\cl}(\kappa) = \lambda = 2^\kappa$.
\end{thm}

Moreover, we extend the above result to a certain class of $\kappa$-complete filters on $\kappa$ (see Theorem~\ref{thm.general.F}). The consistency of $\mf p(\kappa)<\mf b(\kappa) (=\mf{p}_{\cl}(\kappa))$ is originally due to  Shelah and Spasojevi\'{c} \cite{shelah2002cardinal}, however our techniques significantly differ from theirs: We add $\kappa$-Cohen reals and then successively diagonalise the club-filter while preserving a Cohen witness to $\mathfrak{p}(\kappa) = \kappa^+$. We conclude the paper with a list of interesting remaining open questions and a short appendix containing proofs and related results that did not quite fit in earlier sections. 

%Furthermore, we will look at various strategies to prove the consistency of  $\mf p(\kappa)<\mf p_{\cl}(\kappa)$. In particular, we %investigate when a witness for $\mf p(\kappa)$ can be preserved along an iteration that diagonalizes the club filter on $\kappa$. %\todo{Add details}
%We end our paper with a list of open problems.
%Before proceeding, let us remark that  $\mf p_{\cl}(\kappa)$ coincides with a well-known cardinal invariant, .

% \begin{clm}
 
% The cardinal $\mf p_{\cl}(\kappa)$ is the minimal size of an unbounded family of functions from $\kappa^\kappa$. That is, $\mf p_{\cl}(\kappa)=\mf b(\kappa)$.
% \end{clm}

% \medskip

% One of the main questions that motivates this paper takes as a motivation the first item of Theorem \ref{garti_r} and tries to study whether the cardinals $\mf p(\kappa)$ and $\mf t(\kappa)$ can coincide (see Section \ref{pktk}). Although we do not have an answer to this question, we give a duality result inspired in the results of Shelah in \cite{Sh:pt}. Namely, we proved:  

% Finally, this section gives some conditions for the existence and non-existence of certain specific peculiar gaps.

% \bigskip

%Section\ref{pcltcl} on the other hand, studies the relationship between the cardinal $\mf p(\kappa)$ and $\mf p_{\cl}(\kappa)$. 

%more about this part  

\subsection{Notation, terminology and preliminaries} 
We aimed our paper to be self contained. 
For a function $f\in \kappa^\kappa$, we say that $C\subs \kappa$ is $f$-closed if for any $\xi\in C$ and $\zeta<\xi$, $f(\zeta)<\xi$. Note that for any $f$, there are $f$-closed clubs (since $\kappa$ is regular). For a club $C\subseteq \kappa$, we let $\scc{C}$ denote the function $$\scc{C}(\zeta)=\min C\setm (\zeta+1).$$

In forcing arguments, smaller conditions are stronger.

One of the main tools in the study of $\mf p$ has been Bell's theorem: for any $\sigma$-centered poset $\mb P$ and for any collection $\mc D$ of $<\mf p$-many dense subsets of $\mb P$, there is a filter $G\subset \mb P$ that meets each element of $\mc D$. A higher 
analogue of Bell's theorem has been given by Schilhan in \cite{schilhan2018masterarbeit}.

\begin{dfn}[Directed and $\kappa$-specially centered posets]\hfill
\begin{itemize}
    \item A subset $C \subseteq \mathbb{P}$ is called \emph{$\kappa$-directed}, if given $D \in [C]^{<\kappa}$, there is a condition $q \in \mathbb{P}$ such that $q \leq p$ for every $p \in D$.
    \item A poset $\mathbb{P}$ is \emph{$\kappa$-centered} if there exists a sequence $\{C_\gamma : \gamma<\kappa\}$ of $\kappa$-directed subsets of $\mathbb{P}$ so that $\mathbb{P}= \bigcup_{\gamma<\kappa} C_\gamma$.
    \item Assume $\mathbb{P}$ is $<\!\kappa$-closed and $\kappa$-centered, say $\mathbb{P} = \bigcup_{\gamma < \kappa} C_\gamma$ where all $C_\gamma$ are $\kappa$-directed. Say that $\mathbb{P}$ is $\kappa$-\emph{centered with canonical lower bounds} if there is a function $f = f^{\mathbb{P}} : \kappa^{<\kappa} \to \kappa$ such that whenever $\lambda < \kappa$ and  $(p_\alpha : \alpha < \lambda )$ is a decreasing sequence with $p_\alpha \in C_{\gamma_\alpha}$, then there is $p \in C_\gamma$ with $p \leq p_\alpha$ for all $\alpha < \lambda$ and $\gamma = f(\gamma_\alpha : \alpha < \lambda)$.
\end{itemize}
\end{dfn}

For convenience of the reader, we state the higher analogue of Bells theorem mentioned above, as it appears an important tool in the analysis of  $\mf p(\kappa)$ and $\mf t(\kappa)$.

\begin{thm}\label{gbell}%\cite[Theorem 4.3.3]{J:MT}\label{gbell} 
Let $\kappa^{<\kappa}=\kappa$.
Assume $\mathbb{P}$ is a $\kappa$-centered poset with canonical lower bounds and below every $p \in \mathbb{P}$, there is a $\kappa$-sized antichain. Then for any collection $\mc D$ of $<\mf p(\kappa)$-many dense subsets of $\mb P$, there is a filter $G\subset \mb P$ that meets each element of $\mc D$.
\end{thm}

%In the theorem above, the hypothesis on the number of dense sets $\lambda$ that one considers can be dropped. Using the result of Garti \ref{garti_r} that the cofinality of $\mf p(\kappa)$ is $>\kappa$, it is possible to construct a club of cardinals $\lambda$ below $\mf p(\kappa)$ such that $\lambda^{<\kappa}=\lambda$ and so, given a sequence of $\lambda'$ may dense sets, we can extend it (with dummy dense sets of $\mathbb{P}$) to a sequence of length $\lambda> \lambda'$ with the desired properties.

\subsection{Acknowledgments}
The authors would like to thank the Austrian Science Fund (FWF) for the generous support through START Grant Y1012-N35 (Fischer, Montoya and Schilhan ), Grant I4039 (Fischer) and Grant I1921 (Soukup). The last author was also supported by  NKFIH OTKA-113047.

\section{On $\mf p(\kappa),\mf t(\kappa)$ and gaps}\label{pktk}

In their seminal work, Malliaris and Shelah \cite{MaSh:CS} proved that the classical cardinal invariants $\mf p$ and $\mf t$ coincide, answering a longstanding open problem. By now, various interpretations of their proof surfaced (see \cite{Roccasalvo, fremlinpt, casey2017notes,schilhan2018masterarbeit, ulrich2018streamlined}) 
and we shall outline an argument for $\mf p=\mf t$ to motivate our results presented here.

First, we need two notions of gaps. Let $\bar y= (y_\alpha: \alpha<\lambda), \bar x = (x_\beta: \beta<\kappa)$ be sequences from $\oo^\oo$. We say that $(\bar y, \bar x)$ is a \textit{pre-gap} if for every $\gamma< \alpha<\lambda$ and  $\delta< \beta<\kappa$, $$y_\gamma <^{\ast} y_\alpha <^* x_\beta <^{\ast} x_\delta.$$

\begin{dfn}[Tight gaps]
 We call $(\bar y, \bar x)$  a $(\lambda,\kappa)$-\emph{tight gap} if it is a pre-gap and for any $z\in \oo^\oo$:
 $$\hbox{if for all } \alpha<\lambda, y_\alpha  \leq^\ast z\hbox{ then there is }\beta<\kappa\hbox{ such that }x_\beta \leq^\ast z.$$ 
\end{dfn}

\begin{dfn}[Peculiar gaps]
A pre-gap $(\bar y, \bar x)$  is a $(\lambda,\kappa)$-\emph{peculiar gap} if for all $A \in [\omega]^\omega$ and $z \in \omega^A$, 
$$\hbox{if for all }\alpha<\lambda, y_\alpha \restriction A \leq^\ast z\hbox{ then there is }\beta<\kappa\hbox{ such that }x_\beta \restriction A \leq^\ast z.$$ 
\end{dfn}

In other words, a peculiar gap is a pre-gap which is tight everywhere.\\

We give a short outline of the proof of $\mf p = \mf t$. We shall inductively aim to build a tower from a witness to $\mf p$ using the following notion.

\begin{dfn}
Let $\mathcal{A}$ be a family of subsets of $\omega$ with the SIP and let $\mathcal{B}$ be an $\subseteq^\ast$-decreasing sequence of subsets of $\omega$, such that every element of $\mathcal{B}$ has infinite intersection with all $A \in \mathcal{A}$ (write $\mathcal{B} \parallel \mathcal{A}$). We say that $\mathcal{B}$ is a \emph{pseudo-parallel} of $\mathcal{A}$ if there is a pseudo-intersection of $\mathcal{B}$ that has infinite intersection with all elements of $\mathcal{A}$.
\end{dfn}

\begin{lemma}\label{ptl}\hfill
\begin{enumerate}
    \item   (Malliaris, Shelah \cite{MaSh:CS}) If $\mathcal{A}= \{A_\alpha: \alpha<\kappa\}$ is \underline{not} a pseudo-parallel of $\mathcal{B}= \{B_\beta: \beta< \mf p \}$ for $\kappa< \mf p$, then there exists either a tower of length $\mf p$ or a $(\mf p,\kappa)$-peculiar gap. %Malliaris- Shelah (2016).
    \item   (Shelah \cite{Sh:pt}) If there is a $(\mf p,\kappa)$-peculiar gap, then there is a tower of length $\mf p$.
    % \todo{Is there an analogue of this for gaps of slaloms?}
\end{enumerate}
\end{lemma}

Let $(A_\alpha)_{\alpha < \mf p}$ be a family of subsets of $\omega$ witnessing $\mf p$ that  is additionally closed under finite intersections. Define a sequence of sets $B_\alpha$ as follows. Let  $B_0= A_0$ and suppose we have constructed $\mathcal{B}_\beta= \{B_\alpha: \alpha<\beta\}$ for some $\beta< \mf p$ such that $\mathcal{B}_\beta \parallel \mathcal{A}$. If $\beta$ is a successor ordinal $\eta +1$ put $B_{\beta}=B_\eta \cap A_{\beta}$. Then $\mathcal{B}_{\beta+1} \parallel \mathcal{A}$. If $\beta$ is a limit ordinal and $\mathcal{B}_{\beta}$ is a pseudo-parallel of $\mathcal{A}$, take $B$ be a witness for this property and put $B_\beta = B \cap A_\beta$.

Then, we have the following cases: either it is possible to carry the construction along $\mf p$-many steps, in which case the family $\{B_\alpha: \alpha<\mf p\}$ is a tower of length $\mf p$; or there is some ordinal $\beta< \mf p$ (which we can assume is regular) such that the family $\mathcal{B}_\beta= \{B_\alpha: \alpha<\beta\}$ is \underline{not} a pseudo-parallel of $\mathcal{A}$. Then, by Lemma \ref{ptl}, there is a tower of size $\mf p$, which finishes the proof. 

The following results are motivated by the question whether $\mf p(\kappa) = \mf t(\kappa)$ holds for an uncountable cardinal $\kappa$. Theorem~\ref{mainpkappa} below is a generalized version of Lemma \ref{ptl} (1) for uncountable cardinals. 
%We still lack the right analogue of Lemma \ref{ptl} (2), unfortunately.

\begin{dfn}[Slaloms]\label{def.D-supported.slalom}
$\hbox{ }$
\begin{enumerate}
    \item  Suppose that $\mc D\subs [\kappa]^\kappa$ is a $<\kappa$-closed filter. A \emph{$\mc D$-supported slalom} is a map $u:X\to [\kappa]^{<\kappa}\setm\{\emptyset\}$ so that $X\in \mc D$.  We also say that $u$ is an $X$-based slalom.
%    \item A club-supported slalom $u$ is continuous if 
%    $u(\gamma)=\bigcup_{\alpha\in \gamma\cap \dom(u)} %u(\gamma)$ 
    % for every limit point $\gamma\in\dom(u)$.
    \item If $u$ is a $\mc D$-supported slalom, then 
    let $\hbox{set}(u)=\bigcup_{\xi\in\dom(u)} u(\xi)$.
    \item Whenever $u,v$ are $\mc D$-supported slaloms and for all  but $<\kappa$ many $\xi\in \dom u\cap \dom v$, $u(\xi)\subseteq v(\xi)$, we write  $u\subseteq^*v$.
    \end{enumerate}
\end{dfn}

\begin{dfn}(Gaps of $\mc D$-supported slaloms)\label{def.D-supported.gaps} A \emph{$\mc D$-supported $(\mu,\lambda)$-gap of slaloms} is a pair of two sequences $(u_\gamma)_{\gamma<\mu}$ and $(v_\alpha)_{\alpha<\lambda}$ of $\mc D$-supported  slaloms so that 
\begin{enumerate}
    \item for any $\gamma<\gamma'<\mu$ and $\alpha<\alpha'<\lambda$,
$$ u_\gamma\subseteq^* u_{\gamma'}\subseteq^*v_{\alpha'}\subseteq^* v_\alpha,$$
    \item there is no $\mc D$-supported slalom $w$ so that for all $\gamma<\mu$ and $\alpha<\lambda$, $$u_\gamma\subseteq^* w\subseteq^* v_\alpha.$$
\end{enumerate}
\end{dfn}

With this, we are ready to state our main theorem.

%\begin{thm}\label{mainpkappa}
%Suppose that $\kappa^{<\kappa}=\kappa$ is regular and $\mc %D$ is a $<\kappa$-closed filter on $\kappa$. Either $\mf %p(\kappa)=\mf t(\kappa)$ or there is a $\lambda<\mf %p(\kappa)$ and a $\mc D$-supported $(\mf %p(\kappa),\lambda)$-gap of slaloms.
%\end{thm}
%\medskip

\begin{thm}\label{mainpkappa}
Let $\kappa$ be a regular cardinal such that 
$\kappa^{<\kappa}=\kappa$. Either $\mf p(\kappa)=\mf t(\kappa)$ or there is a $\lambda<\mf p(\kappa)$ and club-supported $(\mf p(\kappa),\lambda)$-gap of slaloms.
\end{thm}
\begin{tproof}
Suppose that $(A_\alpha)_{\alpha<\mf p(\kappa)}$ is a family with the SIP but no pseudo-intersection. Let $E_\gamma$ denote a pseudo-intersection for $(A_\alpha)_{\alpha\leq \gamma}$ for $\gamma<\mf p(\kappa)$. Further, suppose that $\mf p(\kappa)<\mf t(\kappa)$.

\begin{clm}
There is a club $X\subs \kappa$ so that for all $\gamma<\mf p(\kappa)$ and almost all $\xi\in \kappa$, $$E_\gamma\cap [\xi,s_X(\xi))\neq \emptyset.$$
\end{clm}
\begin{cproof}
For each $\gamma$, let $X_\gamma$ be the set of accumulation points of $E_\gamma$. Then $X_\gamma$ is a club in $\kappa$ and for all $\xi\in \kappa$, $E_\gamma\cap [\xi,s_{X_\gamma}(\xi))\neq \emptyset$.
%For each $\gamma$, we can find a club $X_\gamma\subs \kappa$ %so that for all $\xi\in \kappa$, $E_\gamma\cap %[\xi,s_{X_\gamma}(\xi))\neq \emptyset$, the club $X_\gamma$ %corresponds to the set of limit points of $E_\gamma$. 
Since $\mf p(\kappa)<\mf t(\kappa)\leq \mf t_{cl}(\kappa)=\mf p_{cl}(\kappa)$ (see Observation \ref{obs:pcl}), we can find a single club $X$ that is a pseudo-intersection of $(X_\gamma)_{\gamma<\mf p(\kappa)}$.
\end{cproof}

Let us try and build sequences $\{B_\alpha\}_{\alpha< \mf p(\kappa)}$, $\{Y_\alpha\}_{\alpha<\mf p(\kappa)}$  so that for each $\beta<\mf p(\kappa)$,

% $\lambda\leq\mf p(\kappa)$ and for each $\alpha$, $Y_\alpha$ is a club such that:
\begin{enumerate}
\item $Y_\beta$ is a club,
    \item $B_\beta\subs^* B_\alpha$ and  $Y_\beta\subs^* Y_\alpha$ for all $\alpha<\beta$,
    \item $B_\beta\subs^* A_\beta$, and 
    \item\label{it:sip} for all $\gamma<\mf p(\kappa)$ such that  $\beta\leq\gamma$, 
    $$\bigcup_{\xi\in Y_\beta}E_\gamma\cap [\xi,s_{X}(\xi))\subs^* B_\beta.$$
\end{enumerate}

%Note that condition (\ref{it:sip}) preserves the SIP. %Indeed:
%\medskip

We could not succeed in constructing such a sequence of length $\mf p(\kappa)$, as otherwise $\{B_\alpha\}_{\alpha<\mf p(\kappa)}$ would be a tower of length $\mf p(\kappa)<\mf t(\kappa)$ without pseudo-intersection. First, note that the SIP is still preserved at any intermediate stage.

\begin{clm} The sequence 
 $\{B_\alpha\}_{\alpha<\lambda}$ has the SIP.
\end{clm}
\begin{cproof}
Suppose that $I\in [\lambda]^{<\kappa}$. Then $Y=\bigcap_{\rho\in I}Y_\rho$ is a club and for any $\gamma\in\mf p(\kappa)\setm \sup I$, the set $\bigcup_{\xi\in Y}E_\gamma\cap [\xi,s_{X}(\xi))$ has size $\kappa$ and is a pseudo-intersection to $\{B_\rho\}_{\rho\in I}$.
\end{cproof}
 Moreover, we can only fail at some limit step $\beta<\mf p(\kappa)$ along the construction. Indeed, if $\beta< \mf p(\kappa)$ and both $B_\beta$ and $Y_\beta$ have been already constructed we can put $B_{\beta+1} = B_\beta \cap A_{\beta+1}$ and $Y_{\beta+1}= Y_\beta$.
%\medskip

Fix this $\beta$ where the induction must fail and lets try to approximate $B_\beta$ and see what goes wrong. First, take some pseudo-intersection club $Z$ to the sequence $\{Y_\alpha\}_{\alpha<\beta}$. 

\begin{lemma} There is a $\subseteq^*$-increasing sequence of slaloms $$\{u_\rho\}_{\beta\leq\rho<\mf p(\kappa)}\subseteq \prod_{\xi\in Z}\mathcal{P}([\xi,s_X(\xi)))$$ so that   $\dom u_\rho=Z_\rho$ is a club such that for all $\rho$ and all $\alpha<\beta$
$$\bigcup_{\xi\in Z_\rho} E_\rho\cap [\xi, s_X(\xi))\subseteq ^* \hbox{set}(u_\rho)\subseteq^* B_\alpha\cap A_\beta.$$
%There is a sequence $(u_\gamma)_{\beta\leq \gamma < \mf %p(\kappa)}$ so that 
%\begin{enumerate}[(i)]
%    \item $\dom u_\gamma=Z_\gamma\subs ^* Z$ is a club,
%    \item $u_\gamma(\xi)\subs [\xi,s_X(\xi))$ for all %$\xi\in Z_\gamma$,
%    \item $\set{u_\gamma}:=\bigcup_{\xi\in %Z_\gamma}u_\gamma(\xi)\subs^* B_\alpha\cap A_\beta$ for all %$\alpha<\beta$,
%    \item $\bigcup_{\xi\in Z_\gamma}E_\gamma\cap %[\xi,s_X(\xi))\subs^* \set{u_\gamma}$,
%    \item if $\gamma<\gamma'$ then $Z_{\gamma'}\subs^* %Z_\gamma$ and $u_\gamma(\xi)\subs u_{\gamma'}(\xi)$ for %almost all $\xi\in Z_{\gamma'}$.
%\end{enumerate}
\end{lemma}

The intuition is that each slalom $u_\gamma$ gives an approximation for $B_\beta$ by $\hbox{set}(u_\gamma)$ which satisfies condition (\ref{it:sip}) with this fixed $\gamma$.  

\begin{cproof}
The sequence is constructed inductively. Suppose we have defined $\{Z_\rho\}_{\beta\leq\rho<\gamma}$ and  $\{u_\rho\}_{\beta\leq \rho< \gamma}$ for some $\gamma \in \mf p(\kappa) \setminus \beta$. We will try to force to find the next slalom $u_\gamma$.

Let $Z_\gamma^-$ be a club, which is  a pseudointersection of $\{Z_\rho\}_{\beta\leq\rho<\gamma}$ and consider the poset $\mathbb{P}_\gamma$ consisting of all triples $(\nu,\mc Y,n)$ such that 
\begin{enumerate}
    \item $\dom(\nu)\in [Z_\gamma^-]^{<\kappa}$ is closed and $n\in\kappa$,
    \item $\forall\xi\in\dom(\nu)$
    $$E_\gamma \cap [\xi, s_X(\xi))\subseteq \nu(\xi)\subseteq A_\beta\cap [\xi, s_X(\xi)),$$
    \item $\mc Y =\mc Y_0\cup \mc Y_1 \in [\gamma]^{<\kappa}$
    where $\mc Y_0\subseteq [\beta,\gamma)$ and
    $\mc Y_1 \subseteq \beta$, and
    \item if $\xi\in Z^-_\gamma\setm n$ then 
    \begin{equation}
        \bigcup_{\rho\in \mc Y_0} u_\rho(\xi)\subseteq \bigcap_{\rho\in \mc Y_1} B_\rho\cap A_\beta\cap [\xi, s_X(\xi))
    \end{equation}

    and   
    \begin{equation}
    E_\gamma\cap [\xi, s_X(\xi))\subseteq \bigcap_{\rho\in\mc Y_1}B_\rho\cap A_\beta\cap [\xi, s_X(\xi)).
    \end{equation}
    
\end{enumerate}
%Define the poset $\mathbb{P}_\gamma$ whose conditions %correspond to pairs $(\nu, \mc Y)$ satisfying the following %conditions: 
%
%\begin{itemize}
%\item $\nu$ is a partial function, $\nu: Z\to %[\kappa]^{<\kappa}$ with $|\dom(\nu)|<\kappa$;
%    \item $\nu$ is a function defined on $Z$ with bounded %domain (i.e. of size $<\!\kappa$) and range contained in %$[\kappa]^{<\kappa}$.
%    \item For all $\xi \in \dom(\nu)$,
%$$E_\gamma \cap [\xi, s_X(\xi)) \subseteq \nu(\xi)\subseteq %A_\beta \cap [\xi, s_X(\xi)).$$
%%    $\nu(\xi) \subseteq A_\beta \cap [\xi, s_X(\xi))$.
%%    \item For all $\xi\in \dom(\nu)$, $E_\gamma \cap [\xi, %s_X(\xi)) \subseteq \nu(\xi)$.
%    \item $\mc Y = \mc Y_0 \cup \mc Y_1$, where $\lvert \mc %Y \lvert < \kappa$, $\mc Y_0 \subseteq \gamma$ and $\mc Y_1 %\subseteq \beta$.
%\end{itemize}

The extension relation is defined as follows: $(\mu, \mc X,m)\leq (\nu, \mc Y,n)$ iff $\mu\supseteq \nu$, $\mc X\supseteq \mc Y$, $m\geq n$ and for all $\xi\in\dom(\mu)\setminus\dom(\nu)$:
$$\xi>n \textmd{ and }\bigcup_{\rho\in \mc Y_0} u_\rho(\xi)\subseteq \mu (\xi)\subseteq \bigcap_{\rho\in \mc Y_1} B_\rho.$$
%\medskip
%Ordered as follows: $(\mu, \mc X) \leq (\nu, \mc Y)$ if and %only if $\mu \supseteq \nu$, $\mc X \supseteq \mc Y$ and 
%\begin{itemize}
%    \item $\dom(\mu) \setminus \dom(\nu)\subseteq %\bigcap_{\rho\in\mc Y_0}\dom(u_\rho)$
%    \item $\forall\xi\in \dom(\mu) \setminus \dom(\nu)$
%    we have: $$\bigcup_{\xi\in \mc Y_0} %u_\rho(\xi)\subseteq\mu(\xi)\subseteq\bigcap_{\rho\in\mc %Y_1} B_\rho.$$
%\end{itemize}
%for each $\xi \in \dom(\mu) \setminus \dom(\nu)$ and $\rho %\in \mc Y$ it is satisfied:
%
%\begin{itemize}
%    \item If $\rho \in \mc Y_0$, then $\xi \in \dom(u_\rho)$ %and $u_\rho(\xi) \subseteq \mu(\xi)$.
%    \item If $\rho \in Y_1$ then $\mu(\xi) \subseteq %B_\rho$. 
%\end{itemize}
%\medskip

\begin{obs}
For any pair $(\nu,\mc Y)$ which satisfies condition (1)-(3) above and almost all $n\in \kappa$, $(\nu,\mc Y,n)\in \mb P_\gamma$.
\end{obs}
\begin{cproof}
Using the facts that $|\mc Y_0|<\kappa$, $|\mc Y_1|<\kappa$ and $u_\rho(\xi)\subseteq [\xi,s_X(\xi))$ we can find $n(\mc Y_0, \mc Y_1)\in\kappa$ such that for each $\xi\in Z_\gamma^-\setminus n(\mc Y_0, \mc Y_1)$,
$$\bigcup_{\rho\in \mc Y_0} u_\rho(\xi)\subseteq \bigcap_{\rho\in \mc Y_1} B_\rho\cap A_\beta\cap [\xi, s_X(\xi)).$$
Moreover, by the hypothesis on $\{B_\alpha\}_{\alpha<\beta}$ for each $\rho\in \mc Y_1$, 
$\bigcup_{\xi\in Y_\rho} E_\gamma\cap [\xi, s_X(\xi))\subseteq^* B_\rho$. However $Z^-_\gamma\subseteq^* Y_\rho$ for each $\rho\in \mc Y_1$ and $E_\gamma\subseteq^* A_\beta$. Thus we can find $m(\mc Y_1)\in\kappa$ such that for each $\xi\in Z^-_\gamma\setminus m(\mc Y_1)$ we have
$$E_\gamma\cap [\xi, s_X(\xi))\subseteq \bigcap_{\rho\in\mc Y_1}B_\rho\cap A_\beta\cap [\xi, s_X(\xi)).$$
Now, any  $n>\max\{\eta, m(\mc Y_1), n(\mc Y_0, \mc Y_1), \max\dom(\nu)\}$ works.
\end{cproof}

\begin{clm}
The poset $\mathbb{P}_\gamma$ is $\kappa$-specially centered.
\end{clm}

\begin{cproof}
Indeed, by $\kappa^{< \kappa}= \kappa$, $\kappa$-centerdness holds if $<\kappa$-many conditions with the same first coordinate are compatible. In the latter case, we can apply the above observation to see that such conditions do have common lower bounds. 
\end{cproof}
\begin{clm}
The poset $\mathbb{P}_\gamma$ is $<\kappa$-closed with canonical lower bounds.
\end{clm}

\begin{cproof}
If $\{p_i\}_{i<j}$ is a decreasing
sequence of conditions, where $j<\kappa$ and $p_i=(\nu_i,\mc Y_i, n_i)$ then 
let $\nu^- = \bigcup_{i<j}\nu_i,\mc Y = \bigcup_{i<j}\mc Y_i$ and $n= \sup_{i<j} n_i$.  Now extend $\nu^-$ to $\nu$ by defining $$\nu(\xi)= \big(E_\gamma\cap [\xi, s_X(\xi))\big)\cup \bigcup_{\rho\in \mc Y_0} u_\rho(\xi).$$ This triple $(\nu,\mc Y,n)$ is in $\mb P_\gamma$ and defines the canonical lower bound.
%$\{p_i=(\nu_i, \mc Y_i): i<j\}$ where $j< \kappa$, is a %decreasing sequence of conditions in $\mathbb{P}_\gamma$, it %is clear that the condition $p=(\bigcup_{i<j}\nu_i, %\bigcup_{i<j}\mc Y_i)$ is stronger that all $p_i$'s and %works as a canonical lower bound.
\end{cproof}

For each $\rho\in\gamma$ the set $D_\rho =\{(\nu, \mc Y,n) \in \mathbb{P}_\gamma: \rho \in \mc Y\}$ is dense. Indeed, given 
$\rho$ and $(\nu,\mc Y,n)\in\mb P_\gamma$ we can find a large enough $n^*$ above $n$ so that $(\nu, \mc Y\cup\{\rho\}, n^*)$ extends $(\nu,\mc Y,n )$. Furthermore:

\begin{clm} For each $\eta\in\kappa$ the set 
 $D^\eta =\{(\nu, \mc Y,n) \in \mathbb{P}_\gamma: \exists \zeta > \eta (\zeta \in \dom(\nu))\}$ is dense in $\mb P_\gamma$.
\end{clm}
\begin{proof}
%  Fix $\eta\in\kappa$ and $(\nu, \mc Y)\in\mb P_\gamma$. By hypothesis on the sequence $\{u_\rho\}_{\beta\leq\rho<\gamma}$ for each $\rho\in \mc Y_0$ and each $\alpha<\beta$ , $\hbox{set}(u_\rho)\subseteq^* B_\alpha\cap A_\beta$. Using the facts that $|\mc Y_0|<\kappa$, $|\mc Y_1|<\kappa$ and $u_\rho(\xi)\subseteq [\xi,s_X(\xi))$ we can find $n(\mc Y_0, \mc Y_1)\in\kappa$ such that for each $\xi\in Z\setminus n(\mc Y_0, \mc Y_1)$,
% $$\bigcup_{\rho\in \mc Y_0} u_\rho(\xi)\subseteq \bigcap_{\rho\in \mc Y_1} B_\rho\cap A_\beta\cap [\xi, s_X(\xi)).$$
% Moreover by hypothesis on $\{B_\alpha\}_{\alpha<\beta}$ for each $\rho\in \mc Y_1$, 
% $\bigcup_{\xi\in Y_\rho} E_\gamma\cap [\xi, s_X(\xi))\subseteq^* B_\rho$. However $Z\subseteq^* Y_\rho$ for each $\rho\in \mc Y_1$ and $E_\gamma\subseteq^* A_\beta$. Thus we can find $m(\mc Y_1)\in\kappa$ such that for each $\xi\in Z\setminus m(\mc Y_1)$ we have
% $$E_\gamma\cap [\xi, s_X(\xi))\subseteq \bigcap_{\rho\in\mc Y_1}B_\rho\cap A_\beta\cap [\xi, s_X(\xi)).$$
% Take any $\zeta\in Z$ such that $\zeta>\max\{\eta, m(\mc Y_1), n(\mc Y_0, \mc Y_1), \max\dom(\nu)\}$  and define 
For any $\zeta>\max(\eta,n)$, we can define $\mu\supset \nu$ on the set $\dom \nu\cup \{\zeta\}$ by
$$\mu(\zeta)= \big(E_\gamma\cap [\zeta, s_X(\zeta))\big)\cup \bigcup_{\rho\in \mc Y_0} u_\rho(\zeta).$$
Then $(\mu, \mc Y, n)$ belongs to $D^\eta$ and extends $(\nu,\mc Y,n)$.
\end{proof}

%
%\medskip
%Indeed, given $\xi \in Z$ and $(\nu, \mc Y) \in %\mathbb{P}_\gamma$ we can choose $\xi'> \xi$ such that %$E_\gamma \cap [\xi', s_X(\xi')) \neq \emptyset$ and for all %$i \in \mc Y \cap \gamma$, $\xi' \in \dom(u_i)$. Moreover, %by hypothesis on the club $X$ and the fact that $E_\gamma$ %is a pseudo-intersection of $\{ A_\alpha: \alpha< \gamma\}$, %we can find $\xi'$ and $\mu \subseteq A_\beta \cap [\xi', %s_X(\xi'))$ such that for all $\alpha \in (Y \cap \beta)$, %$\mu \subseteq B_\alpha$. Then it is enough to put $\nu' = %\nu \cup \{(\xi', \mu)\}$.
%\medskip
By the generalized Bell's theorem, there is a filter $G \subseteq \mathbb{P}_\gamma$ intersecting all the above dense sets. Thus, we can finally  define 
$$u_\gamma = \bigcup \{ \nu : \exists\mc Y \hbox{ such that } (\nu, \mc Y) \in G\}.$$
Observe that $Z_\gamma=\dom u_\gamma$ is a club subset of $Z_\gamma^-$ and hence a pseudo-intersection of all the other $Z_\beta$ for $\beta<\gamma$.

\end{cproof}
Note how  $\set{u_\gamma}$ is a reasonable candidate for $B_\beta$ (with $Z_\gamma$ playing the role of $Y_\beta$):

\begin{obs}
$\set{u_\gamma}$  is almost contained in $A_\beta$ and all  $B_\alpha$ for $\alpha<\beta$, and also satisfies condition (\ref{it:sip}) for a particular $\gamma$.
\end{obs} 

%At this point, we can find a single pseudo-intersection for %the domains $(Z_\gamma)_{\gamma\in \mf p(\kappa)\setm %\beta}$. To make notation easier, we assume that all the %functions $u_\gamma$ are defined on the club $Z$. 
%We can define %$$r_\gamma(\xi)=E_\gamma\cap [\xi,s_X(\xi))$$ for %$\gamma\in \mf p(\kappa)\setm \beta$ and $\xi\in Z$, and

Finally, let us take a pseudo-intersection club for $(Z_\gamma)_{\beta\leq \gamma<\mf p(\kappa)}$ which we shall call $Z$ again to ease notation.  Now,  we  define 
$$v_\alpha(\xi)=B_\alpha\cap A_\beta\cap [\xi,s_X(\xi))$$  for $\alpha< \beta$ and $\xi\in Z$. In turn, for all $\gamma<\gamma',\alpha<\alpha'$ and almost all $\xi\in Z$,
$$ u_\gamma(\xi)\subs u_{\gamma'}(\xi)\subs v_{\alpha'}(\xi)\subs v_\alpha(\xi).$$

Finally, if there is a club $Y_\beta\subs Z$ and $w(\xi)\subs [\xi,s_X(\xi))$ for $\xi\in Y_\beta$ so that for all $\gamma,\alpha$ and almost all $\xi\in Y_\beta$, $$ u_\gamma(\xi)\subs w(\xi)\subs v_\alpha(\xi),$$ then $B_\beta=\set{w}$ would extend $\{B_\alpha\}_{\alpha<\beta}$. Since this is impossible (the construction of the $B$-sets failed at step $\beta$),  we must have produced a $(\mf p(\kappa),\beta)$-gap of club-supported slaloms.
%\medskip
\end{tproof}
%\bigskip

\section{On the sizes of gaps of slaloms}\label{sec:gaps}

Naturally,  Theorem \ref{mainpkappa} prompts us to study the existence of $(\lambda_1, \lambda_2)$-peculiar gaps more closely. In fact, to prove $\mf p(\kappa)= \mf t(\kappa)$, it would suffice to show that there are no  $\mc D$-supported $(\mf p(\kappa),\lambda)$-gaps of slaloms supported for some filter $\mc D$. We could not prove this yet, however, building on \cite{Sh:pt}, we shall present some weaker statements. %We also show in this section that $\mf p(\kappa)$ is regular.

Propositions \ref{gapprop} together with Theorem~\ref{mainpkappa} show that $\mf p(\kappa)$ is regular. The results 
in this section show that in a certain sense there are no club-supported gaps of slaloms which are small on both sides. However
in Proposition \ref{shortdecr} we show that there are short decreasing sequences of slaloms with no lower bound. Finally, in Theorem \ref{thm:MA}, we see how generalized forms of MA effect the existence of gaps.

\begin{prop}\label{gapprop} 
Suppose $\kappa =\kappa^{<\kappa} \leq \lambda_1, \lambda_2$ are  regular cardinals and that there is a club-supported $(\lambda_1, \lambda_2)$-gap of slaloms. Then $\mf p(\kappa) \leq \max\{\lambda_1, \lambda_2\}$. 
\end{prop}

\begin{tproof}
Let  $(u_\alpha: \alpha<\lambda_1)$ and $(v_\beta: \beta< \lambda_2)$ be a club-supported $(\lambda_1, \lambda_2)$-gap of slaloms and assume $\lambda_2<\mf p(\kappa)$. We can assume all the slaloms are defined on a common club $C$ (by taking a pseudo-intersection for all the domains). We shall find a single $w$ that fills the gap on a club set using the generalized version of Bell's theorem (see Theorem \ref{gbell}).
 
%  %, based on a club $C$, i.e.

% \begin{enumerate}
%     \item $\alpha<\alpha'<\lambda_2,\beta<\beta'<\lambda_1$ and almost all $\xi$ in $C$,
%  $$ u_\alpha(\xi)\subs u_{\alpha'}(\xi)\subs v_{\beta'}(\xi)\subs v_{\beta'}(\xi),$$
%     \item there is no $\mc D$-supported slalom $w$ so that for all $\alpha$, $\beta$, and almost all $\xi \in \dom w$, $$u_\alpha(\xi)\subs w(\xi)\subs v_\beta(\xi).$$
% \end{enumerate}

% We use the generalized version of Bell's theorem (see Theorem \ref{gbell}) to force.

We define a $\kappa$-specially centered poset $\mathbb{Q}$ as follows. Conditions in $\mathbb{Q}$ are triples $q=(s^q, \sigma_1^q,\sigma_2^q)$ where 
\begin{enumerate}
    \item $s^q$ is a partial slalom defined some closed, bounded subset of $C$,
    \item  $\sigma_i^q\in [\lambda_i]^{<\kappa}$ for $i=1,2$, and
    \item  for any $\alpha\in \sigma_1^q,\beta\in \sigma_2^q$ and $\eta>\max \dom s$, $u_\alpha(\eta)\subseteq v_\beta(\eta)$.
\end{enumerate}
% . with $\lvert \dom(s) \lvert< \kappa$ and $\sigma \subseteq \lambda_1$ of size $<\kappa$. 

The order on $\mathbb{Q}$ is defined as follows: We say $p \leq q$ if and only if $s^p \sqsupseteq s^q$, $\sigma_i^p \supseteq \sigma_i^q$ and for all $\eta \in \dom(s^p) \setminus \dom(s^q)$,  $$\bigcup_{\alpha\in \sigma_1^q}u_\alpha(\eta)\subseteq s^p(\eta)\subseteq \bigcap_{\alpha\in \sigma_2^q}v_\beta(\eta).$$ 

\begin{clm}
  $\mathbb{Q}$ is a $\kappa$-closed, $\kappa$-specially centered forcing notion of size $\lambda_2$. 
\end{clm}
\begin{cproof}
For a fixed closed and bounded $s\subset C$, any subset of $\mb Q_s=\{q\in\mb Q:s^q=s\}$ has a canonical lower bound. So the partition 
$$\mb Q=\bigcup\{ \mb Q_s: s\in[C]^{<\kappa}, s\hbox{  club}\}$$ 
%\mb Q=\bigcup_{s\in [C]^{<\kappa}}\mb Q_s$ 
witnesses the claim.
\end{cproof}

% \medskip
% Indeed, without loss of generality the set of first coordinates for our conditions can be put in correspondence with $\kappa^{<\kappa}$. Thus $\mathbb{Q}$ can be written as $\mathbb{Q}= \bigcup_{s \in \kappa^{<\kappa}} \{(s, \sigma) \in \mathbb{Q}: \sigma \subseteq \lambda_1 \}$ and given a decreasing sequence of condition $\{(s_\alpha, \sigma_\alpha): \alpha< \gamma\}$ where $\gamma< \kappa$, the condition $(\bigcup_{\alpha<\gamma}s_\alpha, \bigcup_{\alpha<\gamma}\sigma_\alpha) \in \mathbb{Q}$ and fixing a continuous increasing enumeration $f^{\mathbb{Q}}$ from $\kappa^{<\kappa}$, we can have the canonical lower bound we wanted.
% \medskip

%Consider the following collection of subsets of %$\mathbb{Q}$. For $\eta\in \kappa$, let
%$$ \mathcal{D}_\eta= \{ q \in \mathbb{Q}: \eta <\max \dom %%s^q\}$$
%and for all ordinals $\alpha < \lambda_1$ and %$\beta<\lambda_2$, let
%$$ \mathcal{E}_{\alpha, \beta}= \{ q\in \mb Q: \alpha\in %\sigma_1^q ,\beta\in \sigma_2^q \}.$$

\begin{clm} For each $\eta<\kappa$, $\alpha <\lambda_1$ and $\beta < \lambda_2$ the following sets are dense in $\mb Q$: 
\begin{enumerate}
    \item $\mathcal{D}_\eta= \{ q \in \mathbb{Q}: \eta <\max \dom s^q\}$, and 
    \item $\mathcal{E}_{\alpha, \beta}= \{ q\in \mb Q: \alpha\in \sigma_1^q ,\beta\in \sigma_2^q \}$.
\end{enumerate}
%For any $\eta<\kappa$, $\mc D_\eta$ is dense in $\mb Q$. For %any $\alpha < \lambda_1$ and $\beta<\lambda_2$, $\mc %E_{\alpha,\beta}$ is dense in $\mb Q$.
\end{clm}
\begin{cproof} Fix $q\in \mb Q$, $\eta<\kappa$ and  $\alpha < \lambda_1$, $\beta<\lambda_2$. 
%with $\eta\geq \max \dom s$, we can define $p\in \mb Q$ with $\sigma_i^p=\sigma_i^q$ and $\dom s^p=\dom s^q\cup \{\eta+1\}$ where $s^p(\eta)=\bigcup_{\alpha\in \sigma_1^q}u_\alpha(\eta)$. Now, $p\leq q$ and $p\in \mc D_\eta$.
Let $q'=(s',\sigma_1^q\cup\{\alpha\},\sigma_2^q\cup\{\beta\})$ so that $\dom s'=\dom s\cup \{\mu\}$ and for any $\alpha'\in \sigma_1^q\cup\{\alpha\}$ and $\beta'\in  \sigma_2^q\cup\{\beta\}$, if $\eta>\mu$ then $u_{\alpha'}(\eta)\subseteq v_{\beta'}(\eta)$. Moreover, pick $\mu$ to be above $\eta$  and  define $s'(\mu)= \bigcup_{\alpha\in \sigma_1^q}u_\alpha(\mu)$. Then $q'$ is a condition extending $q$ and $ q'\in \mc D_{\eta}\cap \mathcal{E}_{\alpha, \beta}$, as desired.
\end{cproof}

% $\mathcal{D}_i$ is clearly dense, for each $i < \lambda_1$. On the other hand, given $\alpha < \lambda_2$, $\eta \in C$ and $(s, \sigma) \in \mathbb{Q}$. If $\eta \notin \dom(s)$ or $s(\eta) \nsupseteq u_\alpha(\eta)$, we can extend this condition to $(t, \tau)$, by choosing $\xi > \eta$ such that, for all $i \in \sigma$ we have $u_\alpha(\xi) \subseteq v_ i(\xi)$ (this is possible because of (1)). Then, define $t = s \cup \{ (\xi, u_\beta(\xi ))\}$ and $\tau = \sigma$.
%\medskip
By Theorem \ref{gbell}, we can take a filter $G \subseteq \mathbb{Q}$ which intersects all the dense sets $\{\mathcal{D}_\eta\}_{\eta<\kappa} \cup \{\mathcal{E}_{\alpha,\beta}\}_{(\alpha,\beta)\in \lambda_1\times\lambda_2}$. 
Then $D=\bigcup\{\dom s^q:q \in G\}$ is a club and 
$$w =\bigcup\{s^q:q \in G\}$$ is a slalom with domain $D$.
Fix any $(\alpha,\beta)\in(\lambda_1,\lambda_2)$ and pick 
$q\in \mc E_{\alpha,\beta}\cap G$. Then for any $\eta>\max \dom s^q$, we have $u_\alpha(\eta)\subseteq^*w(\eta)\subseteq^* v_\beta(\eta)$ and so 
$$u_\alpha\subseteq^*w\subseteq^* v_\beta,$$
which finishes the proof.
%such that $w_G(\xi) \subseteq v_\beta(\xi)$ for almost all %$\xi \in C$ and all $\beta< \lambda_1$. Moreover, for almost %all $\xi$'s we also have that $u_\alpha(\xi) \subseteq %w_g(\xi)$, but this contradicts (2).  
\end{tproof}

\begin{cor}\label{pk_regular}
$\mf p(\kappa)$ is regular. 
\end{cor}
\begin{cproof}
This follows immediately from Theorem \ref{mainpkappa} and Proposition \ref{gapprop}. Indeed, if $\mf p(\kappa)=\mf t(\kappa)$ then we are done since the latter is regular. Otherwise, there is a $(\mf p(\kappa),\lambda_1)$-gap of slaloms with $\lambda_1<\mf p(\kappa)$. If $\mf p(\kappa)$ is singular of cofinality $\lambda_0$ then we can shrink the left-hand side of the original $(\mf p(\kappa),\lambda_1)$-gap and get a $(\lambda_0,\lambda_1)$-gap of slaloms. This however, contradicts Proposition \ref{gapprop}.
%If $\cf(\mf p(\kappa))<\mf p(\kappa)$ then we can resemble the construction in the theorem above to construct a sequence of $B_\alpha$'s along $\cf( \mf p (\kappa))$ steps. Just notice that  the generalized Bell's theorem applies at all limit steps $\beta$ below $\cf( \mf p(\kappa))$ and so we can always construct $B_\beta$. Hence, a tower $(B_\alpha)_{\alpha<\mf p(\kappa)}$ can be constructed, a contradiction.
\end{cproof}

Yet another bound on the sizes of gaps is the following.

\begin{prop}\label{gap2prop}
Suppose that $\kappa$ is a regular, uncountable cardinal.
If $\lambda< \mf b(\kappa)$ then there is no club-supported $(\kappa, \lambda)$-gap of slaloms on $\kappa$.
\end{prop}
\begin{tproof} Let $\lambda<\mf b(\kappa)$. Suppose that  $\bar u = (u_\alpha: \alpha< \kappa)$ and $\bar v = (v_\xi: \xi<\lambda)$ are sequences of club-supported slaloms on $\kappa$, $\bar u$ is increasing, $\bar v$ is decreasing and $u_\alpha\subseteq ^* v_\xi$ for all $\alpha<\kappa,\xi<\lambda$.

Let $C_\alpha=\dom u_\alpha$. For any club $C$ which is a subset of the diagonal intersection  $\Delta_{\alpha<\kappa}C_\alpha$, we can define a slalom $w_C$ on $C$ by $$w_C(\beta)=\bigcup_{\alpha<\beta}u_\alpha(\beta).$$
It is clear that $u_\alpha\subseteq^* w_C$ for any $\alpha<\kappa$. 

Given a fixed $\xi<\lambda$, there is a club $D_\xi$ so that $\beta\in D_\xi$ and $\alpha<\beta$ implies that $u_\alpha(\beta)\subseteq v_\xi(\beta)$. The family $\{D_\xi:\xi<\lambda\}$ must have a pseudo-intersection $D$ since $\lambda<\mf b(\kappa)=\mf p_{\cl}(\kappa)$. 

Finally, let $w=w_C$ where $C=D\cap \Delta_{\alpha<\kappa}C_\alpha$. Now, for any $\alpha<\kappa$ and $\xi<\lambda$, $u_\alpha\subseteq^* w \subseteq^* v_\xi$
and so $\bar u,\bar v$ is not a gap.
\end{tproof}

In particular, we proved that any $\kappa$-sequence of club-supported slaloms on $\kappa$ has an upper bound. There is an interesting asymmetry here, as there are short decreasing sequences of slaloms without lower bounds.

%The question that follows the result in Theorem \ref{mainpkappa} is under what conditions can we prove the existence of $(\lambda_1, \lambda_2)$-peculiar gaps or their non-existence. In particular, if one can prove that for some $\kappa$ (might be a large cardinal) there are no $(\lambda, \mf p(\kappa))$-peculiar gaps for all $\lambda<\mf p(\kappa)$, then we get $\mf p(\kappa)= \mf t(\kappa)$.

%We first study the existence or non-existence of certain kinds of gaps. 

\begin{prop}\label{shortdecr}
Suppose that $\kappa=\kappa^{<\kappa}$ is a regular, uncountable cardinal.
\begin{enumerate}
    \item There is a $\subseteq^*$-decreasing, $\kappa$-sequence of club-supported slaloms on $\kappa$ that has no lower bound  supported on a stationary set.
\item For any regular $\kappa\leq \lambda$, there is $\kappa$-specially centered poset $\mb P$ which introduces a decreasing $\lambda$-sequence of club-supported slaloms on $\kappa$ with no lower bound supported on a club.
\end{enumerate}
\end{prop}

\begin{tproof} (1)  The case $\lambda=\kappa$  will be instructive to understand the more general argument of (2). We define the decreasing sequence of slaloms $\bar v = (v_\beta: \beta<\kappa)$ with the following properties
\begin{enumerate}
    \item $v_\alpha:\kappa\to [\kappa]^{<\kappa}\setm\{\emptyset\}$,
    \item for any $\alpha<\beta<\kappa$ and $\eta>\beta$, $v_\alpha(\eta)\supseteq v_\beta(\eta)$,
    \item for any limit $\beta\in \kappa$, $\bigcap_{\alpha<\beta}v_\alpha(\beta)=\emptyset$.
\end{enumerate}
The construction is done in $\kappa$ steps: at step $\beta$, we define $v_\alpha(\beta)$ for $\alpha< \beta$ and $v_\beta\uhr \beta+1$. If $\beta$ is a limit ordinal, then we make sure that the sequence of sets $\{v_\alpha(\beta):\alpha<\beta\}$ is strictly decreasing with empty intersection.  We can pick $v_\beta\uhr \beta+1$ arbitrarily, for example, $v_\beta(\eta)=\{0\}$ for all $\eta\leq \beta$.

If $\beta=\alpha+1$ then again we make sure that  $\{v_{\alpha'}(\beta):\alpha'\leq \alpha\}$ is strictly decreasing and we can pick $v_\beta(\eta)=\{0\}$ for all $\eta\leq \beta$.

Finally, given such a sequence $\bar v$, assume that $w:S\to [\kappa]^{<\kappa}\setm\{\emptyset\}$ and $w\subseteq^* v_\alpha$ for all $\alpha<\kappa$. If $S$ is stationary then we can find a limit $\beta\in S$ so that $\alpha<\beta$ implies that $w(\beta)\subseteq v_\alpha(\beta)$. In turn,  $\bigcap_{\alpha<\beta}v_\alpha(\beta)\neq \emptyset$ and this contradiction finishes the proof.

(2) For a general $\lambda<\mf p(\kappa)$, we will force as follows. Define $\mb P$ to be the set of conditions of the form $p=(s_\alpha^p)_{\alpha\in \sigma^p}$ so that $\sigma^p\in [\lambda]^{<\kappa}$ and there is some $\mu^p<\kappa$ such that $s_\alpha^p:\mu^p\to [\kappa]^{<\kappa}\setm \emptyset$.

Extension in $\mb P$ works as follows: $p\leq q$ if 
\begin{enumerate}
    \item $\sigma^p\supseteq \sigma^q$,
\item for any $\alpha\in \sigma^q$, $s_\alpha^p\supseteq s_\alpha^q$, and
\item for any $\alpha<\beta\in \sigma^q$ and $\eta\in \mu^p\setm \mu^q$, $$s_\beta^p(\eta)\subseteq s_\alpha^p(\eta).$$
\end{enumerate}
The following should be straightforward to check:
\begin{enumerate}[(a)]
    \item $\mb P$ is $\kappa$-specially centered;
    \item $\mc D_\eta=\{p\in \mb P:\eta\leq \mu^p\}$ is dense in $\mb P$;
    \item $\mc E_\alpha=\{p\in \mb P:\alpha\in \sigma^p\}$ is dense in $\mb P$.
\end{enumerate}
So, we can take a generic filter $G\subset \mb P$ and define   $$v_\alpha=\bigcup\{s_\alpha^p:p\in \mb P\}.$$ Observe that if $\alpha<\beta<\lambda$ and $\alpha,\beta\in \sigma^p$ for some $p\in G$ then for any $\eta\geq \mu^p$, $v_\beta(\eta)\subseteq v_\alpha(\eta)$. So, $(v_\alpha)_{\alpha<\lambda}$ is a decreasing sequence of slaloms.

Now, suppose that $\dot w$ is a $\mb P$-name for a slalom defined on a club and $p\force \dot w\subseteq^* v_\alpha$ for all $\alpha<\lambda$. Take an elementary submodel $M\prec H(\theta)$ of size $\kappa_0<\kappa$ with all relevant parameters in $M$. Also, assume that $M^{<\kappa_0}\subset M$.

Construct a decreasing sequence of conditions $(p_\xi)_{\xi<\kappa_0}$ in $M$, so that 
\begin{enumerate}
    \item for any $\xi<M\cap \kappa$, there is $\zeta<\kappa_0$ and $\delta_\zeta\in M\cap \kappa\setm \xi$ with $p_\zeta\force \delta_\zeta\in\dom \dot w$ and $\mu^{p_\zeta}\geq \xi$ too,
    \item there is some $\xi_0<\xi_1<\dots$ with $\eta_n\in M\cap \kappa$ so that $p_{\xi_n}\force$ for all $\eta\geq \eta_n$, $\dot w(\eta)\subset v_{n}(\eta)$.
\end{enumerate}
So, we arranged that $\sup_{\xi<\kappa_0}\eta^{p_\xi}=M\cap \kappa$ and any lower bound $q$ for the sequence $(p_\xi)_{\xi<\kappa_0}$ will force that $\delta=M\cap \kappa\in \dom \dot w$ and $\dot w(\delta)\subset v_{n}(\delta)$ for all $n<\oo$. However, we can find a lower bound $q$ such that $q\force \bigcap_{n<\oo}v_{n}(\delta)=\emptyset$. This contradiction finishes the proof.
\end{tproof}

We wonder if (2) above can be proved without forcing but using $\lambda<\mf p(\kappa)$. We now define another kind of gap notion for slaloms:

\begin{dfn}
Let $(u_\alpha: \alpha<\lambda)$ and $(v_\beta: \beta<\mu)$ be two sequences of slaloms based on the same club set $C \subseteq \kappa$. We say that $\{(u_\alpha: \alpha<\lambda), (v_\beta: \beta<\mu)\}$ is a $(\lambda,\mu)$-\emph{tight gap of slaloms} if the following hold:

\begin{enumerate}
    \item For all $\alpha<\alpha'< \lambda, \beta<\beta' <\mu$ and almost all $\xi$ in $C$,
 $$u_\alpha(\xi)\subs u_{\alpha'}(\xi)\subs v_{\beta'}(\xi)\subs v_{\beta}(\xi),$$
   \item If $w$ is a $C$-supported slalom such that $\forall\beta<\mu(w\subseteq^* v_\beta)$, then there is $\alpha<\lambda$ such that $w\subseteq^* u_\alpha$.
%    \item If $w$ is a $C$-supported slalom $w$ such that for %all $\beta< \mu$ and almost all $\xi \in \dom w$, %$$w(\xi)\subs v_\beta(\xi).$$ Then, there is %$\alpha<\lambda$ so that, for almost all $\xi \in \kappa$ %$w(\xi) \subs u_\alpha(\xi)$.
  \item If $w$ is a $C$-supported slalom such that $\forall\alpha<\lambda(u_\alpha\subseteq^* w)$, then there is $\beta<\mu$ such that $v_\beta\subseteq^* w$.
%  \item Similarly, if $w$ is a $C$-supported slalom $w$ such %that for all $\alpha< \lambda$ and almost all $\xi \in \dom %w$, $$u_\alpha (\xi) \subseteq w(\xi) \subs v_\beta(\xi).$$ %Then, there is $\beta<\mu$ so that, for almost all $\xi \in %\kappa$ $ v_\beta(\xi) \subs w(\xi)$.
\end{enumerate}
\end{dfn}

\begin{quest}
Clearly, if $\{(u_\alpha: \alpha<\lambda), (v_\beta: \beta<\mu)\}$ is a $(\mu, \lambda)$-tight gap of slaloms, then it is a gap. Do these notions coincide?
\end{quest}

For the following result, we will use a higher analogue of Martin's axiom relativized to a certain class of posets. 
In order to do this, we will use the following definitions and results of S. Shelah (see Section 2.2 in \cite{BGS:HC}).

% Now we aim to get an axiom that resembles this situation in our case of interest. 

\begin{dfn}
Let $\kappa$ be an uncountable cardinal and $\mathbb{Q}$ be a forcing notion. We say that $\mathbb{Q}$ is \emph{stationary $\kappa^+$-Knaster} if for every $\{p_i: i< \kappa^+\} \subseteq \mathbb{Q}$ there exists a club $E \subseteq \kappa^+$ and a regressive function $f$ on $E \cap S_\kappa^{\kappa^+}$ such that for any $i,j \in E \cap S_\kappa^{\kappa^+}$, if 
$f(i)=f(j)$ then $p_i$ and $p_j$ are compatible.
%$$ f(i) =f(j) \rightarrow p_i \parallel p_j$$
\end{dfn}

Note that if a poset  is stationary $\kappa^+$-Knaster then it is $\kappa^+$-cc.

\begin{dfn}\label{good}
Let $\kappa$ be an uncountable cardinal. A forcing notion  $\mathbb{Q}$ satisfies the $(\ast_\kappa)$-property, and we say it is \emph{$\kappa$-good-Knaster}, if the following conditions hold:

\begin{enumerate}
    \item $\mathbb{Q}$ is stationary $\kappa^+$-Knaster.
    \item Any countable decreasing sequence of conditions in $\mathbb{Q}$ has a greatest lower bound.
    \item Any two compatible conditions in $\mathbb{Q}$ have a greatest lower bound.
    \item $\mathbb{Q}$ is $<\kappa$-closed.\footnote{In the original definition of Shelah, the requirement is
    somewhat weaker, i.e. that $\mathbb{Q}$ is  $\kappa$-strategically closed.} %define this
\end{enumerate}
\end{dfn}

Finally, we can define our forcing axiom.

%Now, we state the form of our axiom $\MA^\kappa$ where %$\kappa$ is an uncountable cardinal. First, 

%Denote by \emph{$\kappa$-good-Knaster} the class of posets %$\mathbb{P}$ which have property $(\ast_\kappa)$ and are %$<\kappa$-closed.
%\medskip
%change the notation! Ask the others
\begin{dfn}\label{def.MA-good-Knaster} Let $\kappa$ be an uncountable cardinal. We say that $\MA(\kappa\hbox{-good-Knaster})$ holds if and only if for all posets $\mathbb{Q}$ in the class \emph{$\kappa$-good-Knaster} and every collection $\mathcal{D}$ of dense sets of $\mathbb{Q}$ of size $<2^{\kappa}$ there is a filter on $\mathbb{Q}$ intersecting all the sets in $\mathcal{D}$. 

\end{dfn}

In the following, we will exploit the consistency of  $\MA(\kappa\hbox{-good-Knaster})$ stated below.

\begin{thm}\label{thm.MA-good-Knaster} Assume GCH. Let $\kappa$ be a regular cardinal such that $\kappa^{<\kappa}=\kappa$ and $\lambda>\kappa$ such that $\lambda^{<\kappa}= \lambda$. Then, there is a cardinal preserving generic extension in which $2^\kappa= \lambda$ and $\MA(\kappa\hbox{-good-Knaster})$ holds. 
\end{thm}

The proof is presented in the Appendix.\\

We now prove that $\MA(\kappa\hbox{-good-Knaster})$  implies the non-existence of certain kinds of tight gaps of slaloms.
 
\begin{thm}\label{thm:MA}
Suppose that $\lambda$ is a cardinal  so that $\cf(\lambda)> \kappa$, $\lambda^{<\kappa}=\lambda$ and that $\MA(\kappa\hbox{-good-Knaster})$ holds. Then there is no tight $(\lambda,\kappa^+)$-gap of slaloms based on a fixed club set $C \subseteq \kappa$. 
\end{thm}
\begin{proof}
Suppose towards a contradiction that there is a $(\lambda, \kappa^+)$-tight gap of slaloms $\{(u_\alpha: \alpha<\lambda), (v_\beta: \beta<\kappa^+)\}$ based on (without loss of generality) $\kappa$ and define the following forcing notion $\mathbb{Q}$. Conditions in $\mathbb{Q}$ are pairs $p = (\bar{s}, \sigma)$ where:
 \begin{itemize}
     \item $\sigma \subseteq \kappa^+$ and $\lvert \sigma \lvert <\kappa$.
     \item $\bar{s} = (s_i)_{i \in \sigma}$ is a sequence of partial slaloms with common domain, a fixed ordinal $\eta_p< \kappa$.
     \item If $i \in \sigma$, $\xi \in \eta_p$, then  $s_i(\xi) \subseteq v_i(\xi)$.
     \item If $i^\ast= \sup(\sigma)$, then $i^\ast > \lvert \sigma \lvert$.
 \end{itemize}

A condition $q= (\bar{t}, \tau)$ is said to extend the condition  $p=(\bar{s}, \sigma)$ if:
\begin{itemize}
    \item $\tau \supseteq \sigma$.
    \item For all $i \in \sigma$, $t_i \sqsupseteq s_i$.
    \item For all $i < i' \in \sigma$ and  $\xi \in \eta_q \setminus \eta_p$, $t_i(\xi) \subs t_{i'}(\xi)$. 
    %for all $\xi \in \eta_q \setminus \eta_p$.
    \item For all $j \in \tau \setminus \sigma$ and $i \in \sigma$ such that $j<i$, there is $\xi \in \eta_q \setminus \eta_p$ such that $v_i(\xi) \subs t_j(\xi)$.
\end{itemize}
 
We want to use our assumption of $\MA(\kappa\hbox{-good-Knaster})$ for this poset and some (to define) collection of dense sets. 
%First, we prove that it satisfies the stationary %$\kappa^+$-Knaster condition:
 
\begin{clm}
 $\mathbb{Q}$ is stationary $\kappa^+$-Knaster and $<\kappa$-closed.
\end{clm}
\begin{cproof}
Suppose $\mathcal{X}= \{p_\alpha: \alpha < \kappa^+\}$ is a sequence of conditions in $\mathbb{Q}$. We want to show that there is a club $E \subseteq \kappa^+$ and a regressive function $f: E \cap S_\kappa^{\kappa^+} \to \mathcal{X}$ such that, if $f(i)=f(j)$ then  $p_i$ and  $p_j$ are compatible. 
 
First, we use the pigeonhole principle and the $\Delta$-system lemma in order to assume, without loss of generality that for all $\gamma<\kappa^+$ the following hold:
 
 \begin{itemize}
     \item $\eta_p=\eta<\kappa$.
     \item $\lvert \sigma_\gamma \lvert= \lambda^\ast < \kappa$.
     \item $\sigma_\gamma \cap \sigma_{\gamma'}= \epsilon$.
     \item If $\sigma_\gamma= \{ i_{\gamma, l}: l < \lambda^\ast \}$ (increasingly ordered), then $s_l^\gamma = s^\ast_\gamma$ for all $l < \lambda^\ast$. Here and throughout the proof $s^\gamma_l$ denotes  $s_{i_{\gamma,l}}$.
     \item The sequence $i_{\gamma, l}$ is strictly increasing in the first coordinate, for $l \notin \epsilon$.
 \end{itemize}
 
Given $\gamma< \gamma'<\kappa^+$, we now claim that $p_\gamma= (\sigma_\gamma, \bar{s}_\gamma)$ and $p_{\gamma'}= (\sigma_{\gamma'}, \bar{s}_{\gamma'})$ are compatible. If true, we can then define $E=\kappa^+$ and  $f : S_\kappa^{\kappa^+} \to \kappa^+$ to be the constant function with value $0$ and we get the stationary $\kappa^+$-Knaster condition.
 
To prove the claim, choose an ordinal $\zeta \geq \eta$ such that, for each $\xi \geq \zeta$:
$$\{ v_{i_{\rho, l}}(\xi): \rho \in \{\gamma, \gamma'\} \wedge l< \lambda^\ast \}$$ is $\subs$-decreasing (this is possible because the $i_{\gamma, l}$ are increasing and the way the $v$'s are arranged).
 
Moreover, we can choose $\zeta$ so that for all $\xi \geq \zeta$, $\lvert v_{i_{\gamma', \lambda^\ast}}(\xi) \lvert + \lambda^\ast > \lvert v_{i_{\gamma, \lambda^\ast}}(\xi)\lvert$. %Explain why this is necessary
 
Define a condition $q=(\bar{t},\tau)$ as follows: $\tau = \sigma_\gamma \cup \sigma_{\gamma'}$ and $\bar{t}=(t_j)_{j\in\tau}$. 
Put $\zeta= \dom(t_i)$ for all $i$ and recall the enumeration of $\sigma_\gamma$ and $\sigma_{\gamma'}$ we have fixed above.

We consider the following cases:

\begin{itemize}
    \item If $j \in \epsilon$, i.e $j = i_{\gamma, l}$ for $l < \lvert \epsilon \lvert$, then define partial slalom $t_j$ as follows:
    \[
    t_j(\xi) = \left \{\begin{array}{ll}
        s^\gamma_j(\xi) & \text{if } \xi < \eta\\
        v_{j}(\xi) & \text{if } \eta \leq \xi < \zeta
        \end{array} \right.
  \]
  \item If $j = i_{\gamma, l}$, for $\lvert \epsilon \lvert \leq l <\lambda^*$, then define partial slalom $t_j$ as follows:
  
  \[
    t_j(\xi) = \left \{\begin{array}{ll}
        s_j^\gamma(\xi) & \text{if } \xi < \eta\\
        v_{i_{\gamma', l'}}(\xi) & \text{if } \eta \leq \xi < \zeta \text{ and } l'<l \text{ is the supremum so that } v_{i_{\gamma',l'}}\subseteq^\ast v_j 
        \end{array} \right.
  \]
  
  \item If $j = i_{\gamma', l}$, for $\lvert \epsilon \lvert \leq l <\lambda^*$, define analogously as in the item above, i.e.
  
  \[
    t_j(\xi) = \left \{\begin{array}{ll}
        s_j^{\gamma'}(\xi) & \text{if } \xi < \eta\\
        v_{i_{\gamma, l'}}(\xi) & \text{if } \eta \leq \xi < \zeta \text{ and } l'<l \text{ is the supremum so that } v_{i_{\gamma,l'}}\subseteq^\ast v_j \\
        v_j(\xi) & \text{if } \eta \leq \xi < \zeta \text{ and } \{l'<l: v_{i_{\gamma,l'}}\subseteq^\ast v_j\}= \emptyset 
        \end{array} \right.
  \]
\end{itemize}

Then $q \leq p_\gamma$ and $q \leq p_{\gamma'}$. 
\end{cproof}
 
It remains to prove that the poset $\mathbb{Q}$ has properties (2), (3) and (4) from Definition \ref{good}. Note Let $\{p_\alpha\}_{\alpha<\gamma}$ be a $<$-decreasing sequence of conditions in $\mathbb{Q}$, where $p_\alpha= (\bar{s}_\alpha, \sigma_\alpha)$. Then there is a canonical lower bound $p=(\bar{s}, \sigma)$ where $\sigma = \bigcup_{\alpha<\gamma} \sigma_\alpha$ (which is still a set of size $<\kappa^+)$ and $\bar{s}$ is defined as follows: $\bar{s}$ is a sequence of partial slaloms $(s_i)_{i \in \sigma}$ with domain $\eta= \sup_{\alpha<\gamma} \eta_\alpha<\kappa$ such that $s_i(\xi)= \bigcup_{\alpha<\gamma} s^\alpha_i(\xi)$ when $s^\alpha_i(\xi)$ is defined (i.e. when $i \in \sigma_\alpha$). This implies that properties (2) and (4) hold. Property (3) hods, as if $p=(\bar{s}, \sigma)$ and $q=(\bar{t}, \tau)$ are compatible, then a canonical lower bound $r= (\bar{u}, \nu)$ has the form $\nu = \sigma \cup \tau$, while the third and fourth conditions in the definition of our poset determine how $r$ must be defined.   

Since by hypothesis $\MA(\kappa\hbox{-good-Knaster})$ holds, there is a generic $G \subseteq \mathbb{Q}$ intersecting the following dense sets. Let $i \in \kappa^+$ and $\eta< \kappa$.
$$\mathcal{D}_{i,\eta}=\{p \in \mathbb{Q}: \sigma_p \nsubseteq i \wedge \forall q \in \mathbb{Q} \, (q \leq p \rightarrow \sigma_q \subseteq i) \wedge \eta_p \geq \eta\}$$
The generic $G$ adds, first of all an unbounded subset of $\kappa^+$, given by $\Sigma_G = \bigcup\{ \sigma_p : p \in G\}$. Also, it generically adds $\kappa^+$-many slaloms $\{w^i_G: i \in \Sigma_G\}$, where $w^i_G= \bigcup \{ s^p_i : p \in G $ and $ (\bar{s}_p)_i= s^p_i \}$. These slaloms satisfy that for all $i< j \in \Sigma_G$ and for almost all $\xi \in \kappa$ $w^i_G(\xi) \subs w^j_G(\xi)$.

Moreover, we have that for all $i< j \in \Sigma_G$ and for almost all $\xi \in \kappa$ $$w^i_G(\xi) \subs w^j_G(\xi) \subs v_j(\xi) \subs v_i(\xi).$$

Now, using the hypothesis that $\{(u_\alpha: \alpha<\lambda), (v_\beta: \beta<\kappa^+)\}$ is a $(\kappa^+, \lambda)$-tight gap of slaloms, given $i \in \Sigma_G$, we can find  $\alpha(i)< \lambda$ such that, for almost all $\xi \in \kappa$ $w^i_G(\xi) \subs u_{\alpha(i)}(\xi)$.

Let $\alpha^\star = \sup \{ \alpha(i): i \in \Sigma_G\}$.  Then for each $i \in \Sigma_G$ we can find $\eta_i < \kappa$ such that for all $\xi > \eta_i$:
$$w^i_G(\xi) \subs u_{\alpha^\star}(\xi) \subs v_i(\xi)$$
Again, using the pigeonhole principle, we can assume without loss of generality that $\eta_i= \eta^\ast$. Then we can pick a condition  $p=(\sigma, \bar{s}) \in G$ so that $j \in \sigma$ where $j \in \Sigma_G$ and $ \lvert j \cap \Sigma_G \lvert \geq \kappa$ and $\eta_p > \eta^\ast$.

Since $\lvert \sigma \lvert < \kappa$, we can choose $i \in \Sigma_G \cap (j \setminus \sigma)$ and $q= (\tau, \bar{t}) \leq p$ for which $i \in \tau$. Then, by the definition of the forcing $\mathbb{Q}$, there is $\eta^p \leq \zeta < \eta_q$ such that $v_i(\zeta) \subs t^i(\zeta)= w^i_G(\zeta)$. But then we get $v_i(\zeta) \subs w^i_G(\zeta) \subs u_{\alpha^\star}(\zeta) \subs v_i(\zeta)$ which is a contradiction.
\end{proof}

\section{On $\mf p(\kappa)$ and $\mf p_{\cl}(\kappa)$}\label{pcltcl}

The definitions of $\mathfrak{p}(\kappa)$ and $\mathfrak{t}(\kappa)$ invoke all $\kappa$-complete filters (resp. towers) on $\kappa$, without giving any additional structural information. Thus it makes sense to first consider smaller classes of filters that may be better understood. One natural way of classifying $\kappa$-complete filters is to consider larger filters in which they simultaneously embed. This leads to the following definition:

\begin{dfn}
 Let $\mathcal{F}$ be a $\kappa$-complete filter on $\kappa$. Then $$\mathfrak{p}_{\mathcal{F}}(\kappa) := \min \{\vert \mathcal{B} \vert : \mathcal{B} \subseteq \mathcal{F} \wedge \mathcal{B} \text{ has no pseudointersection} \}$$ and $$\mathfrak{t}_{\mathcal{F}}(\kappa) := \min \{\vert \mathcal{T} \vert : \mathcal{T} \subseteq \mathcal{F} \wedge \mathcal{T} \text{ is a tower} \}$$ whenever these are defined.
\end{dfn}

Note that $\mathfrak{p}_{\mathcal{F}}(\kappa)$ is defined exactly when $\mathcal{F}$ has no pseudointersection. One of the most interesting examples is $\mathfrak{p}_{\cl}(\kappa) = \mathfrak{p}_{\mathcal{C}}(\kappa)$ where $\mathcal{C}$ is the club filter on $\kappa$. Our goal in this section is to study to study the relationship of $\mathfrak{p}(\kappa)$ to $\mathfrak{p}_{\cl}(\kappa)$. We start by some straightforward observations.

\begin{obs}\label{obs:pcl}
Let $\mathcal{F}$ be a $\kappa$-complete filter on $\kappa$ such that $\mathfrak{p}_{\mathcal{F}}$ is defined, then 
\begin{enumerate}
  \item $\kappa^+ \leq \mathfrak{p}(\kappa) \leq \mathfrak{p}_{\mathcal{F}}$,
  \item whenever $\mathfrak{t}_{\mathcal{F}}$ is defined, $\mathfrak{p}_{\mathcal{F}} \leq \mathfrak{t}_{\mathcal{F}} \leq \mathfrak{t}(\kappa)$, 
\item $\mf p_{\cl}(\kappa)=\mf t_{\cl}(\kappa)=\mf b(\kappa)$.
\end{enumerate}
\end{obs}
\begin{proof}(1) and (2) follow immediately from the definitions. (3) has been shown in \cite{J:GP}. Let us recall the argument. First note that $\mathfrak{p}_\cl(\kappa)$ as well as $\mathfrak{t}_\cl(\kappa)$ are defined. To see that they are equal, let $\lambda=\mf p_{\cl}(\kappa)$ and suppose that $(C_\alpha: \alpha<\lambda)$ is a family of clubs in $\kappa$ with no pseudointersection of size $\kappa$. Build a sequence $(D_\alpha: \alpha<\lambda)$ of clubs so that $D_\beta$ is club and a pseudointersection of $\mc E_\beta =\{D_\alpha:\alpha<\beta\}\cup \{ C_\alpha:\alpha\leq \beta\}$ (note the closure of a pseudointersection is still a pseudointersection). This is possible, since $\mc E_\beta$ is a family of clubs of size $<\mf p_{\cl}(\kappa)$. Now $(D_\alpha: \alpha<\lambda)$ is a witness for $\mathfrak{t}_\cl(\kappa) = \lambda$. To see that $\mf p_\cl(\kappa) = \mf b(\kappa)$ consider the map the sends a function $f \in \kappa^\kappa$ to $C_f = \{ \alpha < \kappa : f''\alpha \subseteq \alpha \}$ and the map sending a club $C$ to $\scc{C}$.
\end{proof}

The consistency of $\mathfrak{p}(\kappa) < \mathfrak{b}(\kappa)$ was first shown in \cite{shelah2002cardinal}.
%but unfortunately the proof given there is slightly unnatural and puts some unnecessary restrictions. 
The argument for showing that $\mathfrak{p}(\kappa)$ stays small in the generic extension, relies on the following theorem which is the main result of the mentioned paper. 

\begin{thm}
If $\kappa\leq \mu<\mf t(\kappa)$ then $2^\mu=2^\kappa$.
\end{thm}

This theorem mirrors the situation at $\omega$. In order to keep $\mathfrak{p}(\kappa)$ smaller than $\mu$ one only needs to ensure that $2^{\mu}$ will be strictly larger than $2^\kappa$ in the generic extension. Using counting of names it can be seen that this will usually not be a problem (starting with an appropriate ground model). Thus starting from GCH, having regular targets $\mu < \lambda$ for $\mathfrak{p}(\kappa)$ and $\mf b(\kappa)$, we first use Cohen forcing to ensure that $2^{\mu} = \lambda^+$ and then we increase $\mathfrak{b}(\kappa)$ to $\lambda$ with Hechler forcing and simultaneously diagonalize $\kappa$-complete filters of size $< \mu$. In this extension $2^{\mu} > 2^\kappa$ and we have ensured that $\mathfrak{p}(\kappa)$ does not blow up.

We will present a more natural approach that amounts to showing that certain witnesses for $\mathfrak{p}(\kappa)$ can be preserved while increasing $\mf b(\kappa)$. This approach leaves more freedom for cardinal arithmetic. On the other hand, up until now, we only know how to apply it for a construction resulting in a model with $\mathfrak{p}(\kappa)= \kappa^+$.

%The main goal of this section is to prove that $\mf p(\kappa) < \mf p_{\cl}(\kappa)$ can consistently hold. First, we present the necessary background to a well-known forcing notion which is used to diagonalize the club filter on $\kappa$. In order to do this, we first prove that Cohen reals are witnesses for $\mf p$  that are preserved after forcing with Mathias forcing with the club filter. 

Let us introduce the forcing used to increase $\mf b(\kappa)$ (i.e. $\mf p_\cl(\kappa)$) or $\mathfrak{p}_{\mathcal{F}}(\kappa)$ more generally for certain classes of $\mathcal{F}$. This poset has been used greatly in the past.  

\begin{dfn}
Let $\mathcal{F}$ be a base for a $\kappa$-complete filter on $\kappa$. The forcing $\mathbb{M}(\mathcal{F})$ consists of conditions $(a,F)$ where $a \in [\kappa]^{<\kappa}$ and $F \in \mathcal{F}$. The order is defined by $(b,E) \leq (a,F)$ iff $E \subseteq F$ and $b \setminus a \subseteq F$. 
\end{dfn}

\begin{fact}
 $\mathbb{M}(\mathcal{F})$ is $\kappa$-closed and $\kappa^+$-cc (in fact $\kappa$-centered with cannonical lower bounds). 
\end{fact}

In what follows, $\mathcal{C}$ will always refer to the collection of clubs from a specific model, which should always be clear from context. 

Our approach, that we announced earlier, will consist of showing that a $<\kappa$ support iteration of $\mathbb{M}(\mathcal{C})$ will not add a pseudointersection to a previously added collection of (more than $\kappa$ many) Cohen reals. As a warm up and an introduction to the argument we will first show that this the case when forcing with $\mathbb{M}(\mathcal{C})$ once. 

\begin{thm}
\label{thm:diagonce}
Let $\kappa$ be uncountable regular and $\kappa^{<\kappa} = \kappa$. Suppose $\langle y_\alpha : \alpha < \kappa^+ \rangle$ is a sequence of Cohen reals added over $V$ and that $c$ is a $\mathbb{M}(\mathcal{C})$ generic over $V[\bar y]$. Then in $V[\bar y][c]$, the filter generated by $\{ y_\alpha : \alpha < \kappa^+ \}$ has no pseudointersection. 
\end{thm}

We write $\mathbb{C}_{\kappa^+}$ for the $<\kappa$-support product of $\kappa^+$ many copies of $2^{<\kappa}$, the forcing for adding a $\kappa$-Cohen real. Let us first check that,

\begin{lemma}
 Whenever $\langle y_\alpha : \alpha < \kappa^+ \rangle$ is a $\mathbb{C}_{\kappa^+}$ generic sequence, then $\{y_\alpha : \alpha \in \kappa^+ \}$ has the SIP in any further extension by $\kappa$-closed forcing. 
\end{lemma}

\begin{proof}
 Let $\Gamma \in [\kappa^+]^{<\kappa}$ be in any extension of $V^{\mathbb{C}_\kappa}$ by a $\kappa$-closed forcing notion. Then $\Gamma \in V$. By genericity over $V$ we may show that $\bigcap_{\alpha \in \Gamma} y_\alpha$ is unbounded in $\kappa$. More precisely, let $p \in \mathbb{C}_{\kappa^+}$ and $\varepsilon \in \kappa$ be arbitrary. Let $\delta > \sup_{i \in \dom p} (\lth p(i))$ and extend $p$ to $q$ such that $q(i)(\delta) = 1$ for every $i \in \Gamma$. 
\end{proof}

\begin{proof}[Proof of Theorem~\ref{thm:diagonce}]
 In $V[\bar y]$ assume $\dot x$ is a $\mathbb{M}(\mathcal{C})$ name for an element of $[\kappa]^\kappa$. Consider the set $$ X = \{ (a, \alpha) : a \in [\kappa]^{<\kappa}, \alpha < \kappa, \exists C \in \mathcal{C} ((a,C) \Vdash \alpha \in \dot x) \}.$$ Then $X \in V[\langle y_\alpha : \alpha < \delta \rangle]$ for some $\delta < \kappa^+$. We want to show that $\dot x [c] \not\subseteq^* y_\delta$. First recall that $y_\delta$ is in fact Cohen over $V[\langle y_\alpha : \delta \neq\alpha < \kappa^+ \rangle]$. Thus for the proof we may simply assume that $X \in V$ and show that $\dot x [c] \not\subseteq^* y$ where $y$ is Cohen over $V$ and $c$ is $\mathbb{M}(\mathcal{C})$ generic over $V[y]$. 
 
 Suppose in $V[y]$ that $(a,C)$ is an arbitrary condition in $\mathbb{M}(\mathcal{C})$. We have that $a \in V$ and there is some name $\dot C \in V$ so that $\Vdash$``$\dot C$ is club'' in Cohen forcing and $\dot C [y] = C$.
 
 Now suppose that $s \in 2^{<\kappa}$ is an arbitrary condition in Cohen forcing. Now let us define two decreasing sequences $\{p_i^0 : i < \kappa\}$ and $\{p_i^1 : i < \kappa\}$ in Cohen forcing such that the following holds: 
 
 \begin{itemize}
 \item $p_0^0 = p_0^1 = s$, 
  \item if $\bigcup p_i^0 = f_0$ and $\bigcup p_i^1 = f_1$ then $f_0^{-1}(\{1\}) \cap f_1^{-1}(\{1\}) = s^{-1}(\{1\})$, 
  \item the sets $\tilde C^0 = \{ \alpha : \exists i (p_i^0 \Vdash \alpha \in \dot C) \}$ and $\tilde C^1 = \{ \alpha : \exists i (p_i^1 \Vdash \alpha \in \dot C) \}$ are clubs.
 \end{itemize}

 The sequences $\bar p^0$ and $\bar p^1$ are simply interpreting sequences for $\dot C$ below $s$. But we additionally ensure that the sets defined by $\bigcup p^0_i$ and $\bigcup p^1_i$ are disjoint up to their common initial part $s$. Call these sets $y^0 \subseteq \kappa$ and $y^1 \subseteq \kappa$ 
 
 Let $\tilde C = \tilde C^0 \cap \tilde C^1$. Recall that $\tilde C$ will still be club in $V[y]$. Thus there is $b \in [\tilde C]^{<\kappa}$ and $\alpha > \sup \operatorname{dom}(s)$ so that $\min b > a$ and $(a \cup b, \alpha) \in X$. As $\bigcup p^0_i$ and $\bigcup p^1_i$ define disjoint sets there is at least one $j \in 2$ so that $\alpha$ is not in $y^j$. Say wlog $j = 0$. Now we can extend $s$ to some $t = p_i^0$ for some $i$ such that $p_i^0 \Vdash b \subseteq \dot C$, $\alpha \in \operatorname{dom}(t)$ and $t(\alpha) = 0$.

 Thus by genericity we shown that back in $V[y]$ we can extend $(a,C)$ to $(a \cup b, C')$ so that $(a \cup b, C') \Vdash \alpha \in \dot x$ but $\alpha \notin y$. Now by genericity of $c$ we know that $\dot x [c] \not\subseteq^* y$. 
\end{proof}

Now we are going to consider the more general case of iterating $\mathbb{M}(\mathcal{C})$ many times with $<\kappa$-support. For an ordinal $i$ we will write $\mathbb{M}(\mathcal{C})_i$ for the $i$-length $<\kappa$-support iteration of $\mathbb{M}(\mathcal{C})$.

\begin{thm}
 (GCH) For any regular uncountable $\kappa < \lambda$, where $\kappa = \kappa^{<\kappa}$, there is a $\kappa$-closed, $\kappa^+$-cc forcing extension in which $\mathfrak{p}(\kappa) = \kappa^+ <  \mathfrak{p}_{\cl}(\kappa) = \lambda = 2^\kappa$.
\end{thm}

\begin{proof}
 We are going to first add $\kappa^+$ many ($\kappa$-)Cohen reals $\langle y_\alpha : \alpha < \kappa^+ \rangle$ and then iteratively diagonalize the club filter for $\lambda$ many steps. Thus the poset that we are using is $\mathbb{P} = \mathbb{C}_{\kappa^+} * \dot{\mathbb{M}}(\mathcal{C})_{\lambda}$, where $\dot{\mathbb{M}}(\mathcal{C})_{\lambda}$ is a $\mathbb{C}_{\kappa^+}$ name for the $<\kappa$-support iteration of $\mathbb{M}(\mathcal{C})$ of length $\lambda$. This forcing notion is $\kappa$-closed, has the $\kappa^+$-cc and forces $2^\kappa = \lambda$ by a counting argument. Also it is clear that $V^{\mathbb{P}} \models \mathfrak{p}_{\cl}(\kappa) = \lambda$. Thus we are left with showing that $V^{\mathbb{P}} \models \mathfrak{p}(\kappa) = \kappa^+$.
 
 Let us make some remarks on the notation that we will use. 
 
 \begin{itemize}
  \item We will assume that conditions in $\mathbb{M}(\mathcal{C})_\lambda$ always have the form $(\bar a, q)$, where \begin{itemize}
                                                                                                     
                                                                                                      \item $\bar a = \langle a_i : i \in I \rangle$, $I \in [\lambda]^{<\kappa}$, $a_i \in [\kappa]^{<\kappa}$, 
                                                                                                      \item $q$ is a function with $\dom q = I$ and $ q (i)$ is a ${\mathbb{M}}(\mathcal{C})_i$ name for a club for every $i \in I$. 
                                                                     \end{itemize}
                                                                                                     
A pair $(\bar a, q)$ as above is naturally interpreted as the condition $\langle \check{a}_i, {q} (i) \rangle_{i \in I}$. 

\item Similarly we will assume that conditions in $\mathbb{C}_{\kappa^+} * \dot{\mathbb{M}}(\mathcal{C})_{\lambda}$ have the form $(p,\bar a, \dot q)$, where 

\begin{itemize}

\item $p \in \mathbb{C}_{\kappa^+}$
                                                                                                      \item $\bar a \in V$, 
                                                                                                      \item $\dot q$ is a $\mathbb{C}_{\kappa^+}$ name for an object as above.  
                                                                                                     \end{itemize}
                                                                                                     
It is easy to see, using $\kappa$-closure, that conditions of this form are dense in $\mathbb{P}$. 

\item A nice $\mathbb{M}(\mathcal{C})_\lambda$-name $\dot x$ for an element of $P(\kappa)$ has the form $$ \bigcup_{\alpha < \kappa} A_\alpha \times \{ \check \alpha \}$$ where $A_\alpha$ is an antichain in $\mathbb{M}(\mathcal{C})_\lambda$ (thus has size $\leq \kappa$) and for every $(\bar a, q) \in A_\alpha$ and $i \in \dom q$, $q(i)$ is a nice $\mathbb{M}(\mathcal{C})_i$-name. Thus we define nice $\mathbb{M}(\mathcal{C})_i$-names for subsets of $\kappa$ inductively on $i \in \lambda$.

\item It is well known that for any $\mathbb{M}(\mathcal{C})_\lambda$-name $\dot y$ for a subset of $\kappa$, there is a nice name $\dot x$ so that $\Vdash \dot y = \dot x$. 

\item By induction on nice names we see that, $\vert \trcl (\dot x) \vert \leq \kappa$. Namely, assume this is known for nice $\mathbb{M}(\mathcal{C})_i$-names for every $i < j$. Let $\dot x$ be a nice $\mathbb{M}(\mathcal{C})_j$-name. Then $\dot x = \bigcup_{\alpha < \kappa} A_\alpha \times \{ \check \alpha \}$ where each $A_\alpha$ is a set of $\mathbb{M}(\mathcal{C})_j$ conditions of size at most $\kappa$. For each condition $(\bar a, p) \in A_\alpha$, $\vert \dom p \vert < \kappa$. For each $i \in \dom(p)$, $p(i)$ is a nice $\mathbb{M}(\mathcal{C})_i$-name which, by assumption, has transitive closure of size at most $\kappa$.

 \end{itemize}

\begin{clm}
  $\{ y_\alpha : \alpha \in \kappa^+ \}$ has no pseudointersection after forcing with $\mathbb{M}(\mathcal{C})_\lambda$.  
\end{clm}

\begin{proof}
 In $V^{\mathbb{C}_{\kappa^+}} = V[\bar y]$, let $\dot x$ be a nice $\mathbb{M}(\mathcal{C})_\lambda$ name for an element of $[\kappa]^\kappa$. Then there is $\gamma < \kappa^+$, such that $\dot x \in V[\langle y_\alpha : \alpha \in \kappa, \alpha \neq \gamma \rangle]$. We will show that $\dot x$ can not be almost contained in $y_\gamma$. Without loss of generality we may assume that $\dot x \in V$ and that we are adding a single Cohen real $y=y_\gamma$ over $V$ (by putting $V[\langle y_\alpha : \alpha \in \kappa, \alpha \neq \gamma \rangle]$ as the new ground model) and then we are forcing with $\mathbb{M}(\mathcal{C})_\lambda$ in $V[y]$.  
 
 Now suppose that $(p,\bar a, \dot q) \Vdash \dot x \setminus \varepsilon \subseteq \dot y$, where $(p,\bar a, \dot q)\in \mathbb{C} * \dot{\mathbb{M}}(\mathcal{C})_\lambda$ and $\varepsilon \in \kappa$. Let $y$ be $\mathbb{C}$ generic over $V$ with $p$ in the generic filter. Define $y' \in 2^\kappa$ so that $y'(i)= p(i) = y(i)$ for $i \in \dom p$ and $y'(i) = 1 - y(i)$ for $i \in \kappa \setminus \dom p$. It is well known that $y'$ is also generic over $V$ with $p$ in it's generic filter. Moreover $V[y] = V[y'] =: W$. But note that $q := \dot q [y] \neq \dot q[y'] =: q'$ is very much possible. Still in $W$, $(\bar a, q)$ and $(\bar a, q')$ are compatible. Namely we may define $r \colon \dom q \to W$ by putting $r(i)$ a $\mathbb{M}(\mathcal{C})_i$ name for $q(i) \cap q'(i)$. By induction we see that for any $i \in \dom q$, $$(\bar a \restriction i,r \restriction i) \leq (\bar a \restriction i, q \restriction i), (\bar a \restriction i, q' \restriction i) $$ and that $$r \restriction i \Vdash q(i), q'(i) \text{ are clubs.}$$ 
 
 Thus indeed $r(i)$ is a $\mathbb{M}(\mathcal{C})_i$ name for a club, so $(\bar a, r)$ is a condition and $(\bar a, r) \leq (\bar a, q), (\bar a, q')$. Now let $(\bar b, s) \leq (\bar a, r)$ and $\delta \in \kappa \setminus \varepsilon$ so that $$(\bar b, s) \Vdash \delta \in \dot x.$$
 
 Since $y \cap y' \subseteq \varepsilon$, $\delta \notin y$ or $\delta \notin y'$. Say $\delta \notin y$. Then whenever $G$ is $\mathbb{M}(\mathcal{C})_\lambda$ generic over $W$ with $(\bar b, s) \in G$, $W[G]\models \dot x[G] \setminus \varepsilon \not\subseteq y$. At the same time, $(p, \bar a, \dot q)$ is in the corresponding $\mathbb{C} * \mathbb{M}(\mathcal{C})_\lambda$ generic over $V$. This gives a contradiction. Similarly when $\delta \notin y'$. 
\end{proof}

\end{proof}

Analyzing the proof of the above result, we see that this result can be extended to a more general class of filters. 

\begin{thm}\label{thm.general.F}
 (GCH) For any regular uncountable $\kappa < \lambda$, where $\kappa = \kappa^{<\kappa}$, there is a $\kappa$-closed, $\kappa^+$-cc forcing extension in which $\mathfrak{p}(\kappa) = \kappa^+ <  \mathfrak{p}_{\mathcal{F}}(\kappa) = \lambda = 2^\kappa$ for any $\kappa$-complete filter $\mathcal{F}$ on $\kappa$ that is ordinal definable over $H(\kappa^+)$.
\end{thm}

We say that $\mathcal{F}$ is ordinal definable over $H(\kappa^+)$ if there is a formula $\varphi$ in the language of set theory and finitely many ordinals $\alpha_0 < \dots < \alpha_{n-1}< \kappa^+$ so that $$x \in \mathcal{F} \leftrightarrow H(\kappa^+) \models \varphi(x,\bar\alpha).$$

For example, $\mathcal{C}$ is ordinal definable over $H(\kappa^+)$.

\begin{proof}
Let $\langle \varphi_\xi(x,\bar \alpha_\xi) : \xi \in \kappa^+ \rangle$ enumerate all formulas in one free variable $x$ and parameters $\bar \alpha = (\alpha_0, \dots, \alpha_k) \in (\kappa^+)^{<\omega}$ in the language $\{\in\}$. 

As before we first add $\kappa^+$ many Cohen reals using $\mathbb{C}_{\kappa^+}$. Then in $V^{\mathbb{C}_{\kappa^+}}$ we define an iteration $\langle \mathbb{P}_i, \dot{\mathbb{Q}}_i : i < \lambda \rangle$ with $\mathbb{Q}_i = \prod_{\xi < \kappa^+} \mathbb{M}(\mathcal{F}_\xi)$ where $$\mathcal{F}_\xi = \{ x \in [\kappa]^\kappa : H(\kappa^+)^{V^{\mathbb{P}_i}} \models \varphi_\xi(x,\bar \alpha_\xi)  \}$$ if this defines a $\kappa$-complete filter (in $V^{\mathbb{P}_i}$) or $$\mathcal{F}_\xi = \{ \kappa \}$$ else.

Again we consider conditions in $\mathbb{P}_\lambda$, as pairs $(\bar a, q)$ where $\dom a \in [\kappa^+ \cdot \lambda]^{<\kappa}$, $a_{\kappa^+ \cdot i + \xi} \in [\kappa]^{<\kappa}$ and $q$ is a function with domain $\dom a$ so that $q(\kappa^+ \cdot i + \xi)$ is a $\mathbb{P}_i$ name for an element of $\mathcal{F}_\xi$. Similarly we define the notion of nice names. 

It is crucial to note that $\mathbb{P}_\lambda$ only depends on the model $V^{\mathbb{C}_{\kappa^+}}$ and not on the particular set of generic Cohen reals.  Then using the same argument as before we see that $\mathfrak{p}(\kappa) = \kappa^+$ in $V^{\mathbb{C}_{\kappa^+} * \dot{\mathbb{P}}_\lambda}$.   

Now suppose $\mathcal{F}$ is ordinal definable over $H(\kappa^+)$ in $V^{\mathbb{C}_{\kappa^+} * \dot{\mathbb{P}}_\lambda}$ and $\mathfrak{p}_{\mathcal{F}}(\kappa)$ is defined. Say $\mathcal{F}$ is defined by $\varphi_\xi$. Let $\mathcal{B} \subseteq \mathcal{F}$ with $\vert \mathcal{B} \vert < \lambda$. Then there is $i < \lambda$ so that $\mathcal{B} \subseteq V^{\mathbb{C}_{\kappa^+} * \dot{\mathbb{P}}_i}$. Moreover we find $j \geq i$ so that $(H(\kappa^+)_j,\in) \preccurlyeq (H(\kappa^+)_\lambda,\in)$, where $H(\kappa^+)_j = \{ x \in H(\kappa^+) : x \in V^{\mathbb{C}_{\kappa^+} * \dot{\mathbb{P}}_j} \}$. To see this just note that $\vert H(\kappa^+)_i \vert < \lambda$ for every $i < \lambda$. Thus we can find the $< \lambda$ many required Skolem-witnesses over $H(\kappa^+)_i$ in $H(\kappa^+)_{S(i)}$ for some $S(i) < \lambda$. Applying $S$ recursively $\kappa^+$ many times, by taking suprema at limits, yields the desired situation (since no new elements of $H(\kappa^+)$ are introduced in limits of cofinality $\kappa^+$). In $V^{\mathbb{C}_{\kappa^+} * \dot{\mathbb{P}}_j} $, $\mathcal{F}_\xi$ is a $\kappa$-complete filter on $\kappa$ with $\mathcal{B} \subseteq \mathcal{F}_\xi$ and $\mathbb{Q}_j$ adds a pseudointersection to $\mathcal{B}$.  
\end{proof}

\begin{thm}
\label{thm:indestructible}
 ($\mf p(\kappa) = 2^\kappa$) Let $\mathcal{P}$ be a collection of $\kappa^+$-cc forcing notions, each of size $\leq 2^\kappa$ and $\vert \mathcal{P} \vert \leq 2^\kappa$. Then there is a tower which is indestructible by any $\mathbb{P} \in \mathcal{P}$. 
\end{thm}

\begin{lemma}\label{lem:indestructible} Let $\mf p(\kappa) = \lambda$. There is a map $\varphi \colon 2^{<\lambda} \to [\kappa]^\kappa$ so that for each $f \in 2^{\lambda}$, $\langle \varphi(f\restriction \alpha) : \alpha < \lambda \rangle$ is a tower and $\varphi(s^\frown 0) \cap \varphi(s^\frown 1) = \emptyset$ for every $s \in 2^{<\lambda}$.
\end{lemma}

\begin{proof}
 See the proof of Theorem 7 in \cite{shelah2002cardinal}. 
\end{proof}

\begin{proof}[Proof of Theorem~\ref{thm:indestructible}]
 Let $\varphi$ be as in Lemma~\ref{lem:indestructible} and $\lambda = 2^\kappa$. Recall that if $\mathbb{P}$ is $\kappa^+$-cc then we can assume that all $\mathbb{P}$ names for elements of $[\kappa]^\kappa$ are of size at most $\kappa$. Enumerate all triples $\langle \mathbb{P}_\alpha, p_\alpha, \dot x_\alpha : \alpha < \lambda \rangle$ where $\mathbb{P}_\alpha \in \mathcal{P}$, $p_\alpha \in \mathbb{P}_\alpha$ and $\dot x_\alpha$ is a $\mathbb{P}$ name for an element of $[\kappa]^\kappa$. We recursively define $f \in 2^{\lambda}$ as follows: 
 
 Given $s_\alpha \in 2^\alpha$, let $y_0 = \varphi(s^\frown 0)$ and $y_1 = \varphi(s^\frown 1)$. As $y_0 \cap y_1 = \emptyset $ we have that $p_\alpha \Vdash \dot x_\alpha \subseteq^* y_0 \wedge x_\alpha \subseteq^* y_1$ is impossible. Thus for some $i \in 2$ we have that there is $q_\alpha \leq p_\alpha$ so that $q_\alpha \Vdash \dot x_\alpha \not\subseteq^* y_i$. Let $s_{\alpha+1} = s^\frown i$. At limits we let $s_\alpha = \bigcup_{\xi < \alpha} s_\xi$. Finally $f := \bigcup_{\alpha < \lambda} s_\alpha$. 
 
 The tower defined by $f$ is as required. Namely given $\mathbb{P} \in \mathcal{P}$, $p \in \mathbb{P}$ and $\dot x$ a $\mathbb{P}$-name for an unbounded subset of $\kappa$, say $(\mathbb{P}, p, \dot x) = (\mathbb{P}_\alpha, p_\alpha, \dot x_\alpha)$, we have that $q_\alpha \leq p_\alpha$ forces that $\dot x$ is not almost contained in $\varphi(s_\alpha)$. 
\end{proof}

\section{Questions and problems}

Here, we collect some of the natural open problems that occurred during our project. Most notably:

\begin{quest}
Is $\mf p(\kappa)=\mf t(\kappa)$ for any infinite $\kappa$? 
\end{quest}

Maybe something easier would be the following.

\begin{quest}Does $\mf p(\kappa)=\mf t(\kappa)$ hold for  a measurable or (weakly) compact cardinal $\kappa$?
\end{quest}

To approach this problem, one might try to answer the following.

\begin{quest}
Suppose that there is a $(\mf p(\kappa),\lambda)$-gap of club-supported slaloms for some $\lambda<\mf p(\kappa)$. Is there a tower of size $\mf p(\kappa)$ necessarily?
\end{quest}

It also remains open if the notion of gaps and tight gaps are the same for club-supported slaloms.\\

In addition to the club filter, one may define $\mf p_{\mc U}$ for any sub-collection $\mc U\subset [\kappa]^\kappa$.

\begin{quest}
Suppose that $\kappa$ is a measurable cardinal with a $<\kappa$-closed, normal ultrafilter $\mc U$. Does $\mf p_{\mc U}<\mf p_{cl}$?
\end{quest}

Even for $\kappa=\omega$, it would be interesting to construct a (large) collection of ultrafilters $(\mc U_\xi)_{\xi\in I}$ so that the corresponding cardinals $\mf p_{\mc U_\xi}$ are all distinct.

\section{Appendix}

\subsection{Other higher analogues of $\mf p$ and $\mf t$}\label{app:1}
Imposing the SIP property ensures that  $\mf p(\kappa)$ and $\mf t(\kappa)$ fall into the interval $[\kappa^+, 2^\kappa]$. Unpublished work Brian and Verner examines another generalization of the pseudo-intersection and tower numbers  to $\kappa$. Consider the cardinals $\mf p^\ast(\kappa)$ and $\mf t^\ast(\kappa)$ defined below:
%\begin{dfn}[Other higher analogues to $\mf p(\kappa)$ and $\mf t (\kappa)$] \hfill
\begin{itemize}
 \item $\mf p^\ast(\kappa) = \min \{ \lvert \mathcal{F} \lvert: \mathcal{F} $ is a family with the \underline{finite} intersection property and no pseudo-intersection of size $\kappa \}$.
 \item $\mf t^\ast(\kappa) = \min \{ \lvert \mathcal{T} \lvert: \mathcal{T} $ is a well-ordered family of subsets of $\kappa$ without pseudo-intersection of size $\kappa \}$.
\end{itemize}
%\end{dfn}
Here the \emph{finite intersection property} refers to the fact that for any finite sub-family $\mathcal{F}' \subseteq \mathcal{F}$, $\bigcap \mathcal{F}'$ has size $\kappa$.

\begin{prop}[Brian-Verner]\hfill
\begin{itemize}
    \item If $\kappa$ is a cardinal with uncountable cofinality, then $\mf p^\ast(\kappa) = \mf t^\ast(\kappa)= \aleph_0$. 
    \item If $\kappa$ is an uncountable cardinal with $\cf(\kappa)= \omega$, then $\mf t^\ast(\kappa)$ is uncountable. 
    \item If $\kappa$ is an uncountable cardinal with $\cf(\kappa)=\omega$, then $\mf p^\ast(\kappa)= \omega_1$.
\end{itemize}
\end{prop}
\medskip

\subsection{Consistency of $\hbox{MA}(\kappa\hbox{-good-Knaster})$}

Finally, we present the proof of the generalized Martin's Axiom for posets with property $(\ast_\kappa)$ that we applied in Section \ref{sec:gaps}. The proof is based on the following iteration theorem but otherwise resembles the classical proof of Martin's Axiom.

\begin{thm}\label{thm.shelah.goodKnaster}(Shelah, 1976; see~\cite{Sh:Baumhauer}).
Let $\kappa$ be an uncountable cardinal and $(\mathbb{P}_\alpha, \dot{\mathbb{Q}}_\alpha: \alpha<\delta)$ be a $<\kappa$-support iteration such that for every $\alpha< \delta$:
$$\Vdash_{\mathbb{P}_\alpha} \dot{\mathbb{Q}}_\alpha \hbox{ satisfies property }(\ast_\kappa)$$
Then $\mathbb{P}_\delta$ is stationary $\kappa^+$-Knaster.
\end{thm}
%\todo{Referece!!}

\begin{proof}[Proof of Theorem \ref{thm:MA}]
We define a $<\kappa$-support iteration $(\mathbb{P}_\alpha, \dot{\mathbb{Q}}_\alpha: \alpha< \lambda)$ such that for all $\alpha<\lambda$:
 
\begin{itemize}
     \item $\Vdash \dot{\mathbb{Q}}_\alpha$ is has the property $(\ast_\kappa)$.
     \item $\Vdash \lvert \dot{\mathbb{Q}}_\alpha \lvert <\lambda$.
\end{itemize}
 
Since by theorem~\ref{thm.shelah.goodKnaster} the poset  $\mathbb{P}= \mathbb{P}_\lambda$ is stationary $\kappa^+$-Knaster condition and  it is  $<\kappa$-closed,
$\mb P$ preserves cardinals. Also, since $\lambda$ is regular and $\lambda^\kappa = \lambda$, we have $\lvert \mathbb{P}_\alpha \lvert \leq \lambda$.

Define $\dot{\mathbb{Q}}_\alpha$ by induction on $\alpha<\lambda$ as follows. Fix a bookkeeping function $\pi: \lambda \to \lambda \times \lambda$ such that $\pi(\alpha)=(\beta, \gamma)$ implies $\beta \leq \alpha$. If we have defined $\dot{\mathbb{Q}}_\beta$ for all $\beta < \alpha$ and $\pi(\alpha)=(\beta, \gamma)$, we can look at the $\gamma$-th $\mathbb{P}_\beta$-name $\dot{\mathbb{Q}}$ in $V^{\mathbb{P}_\beta}$ for a poset of size $<\lambda$ with  the property $(\ast_\kappa)$. Define $\mathbb{Q}_{\alpha}= \dot{\mathbb{Q}}$.
 
First, we will show that that 
$V^{\mathbb{P}} \models \MA^\kappa(\kappa\hbox{-good-Knaster}_{<\lambda}) \wedge  2^\kappa=\lambda$, where  $\MA(\kappa\hbox{-good-Knaster}_{<\lambda})$ is the restriction of $\MA(\kappa\hbox{-good-Knaster})$ to posets of cardinality stricly smaller than $\lambda$.
 
Let $\dot{\mathbb{R}}$ be a $\mathbb{P}$-name for a  poset with property $(\ast_\kappa)$ such that $\Vdash_{\mathbb{P}} \lvert \dot{R} \lvert < \lambda$ and let $\dot{\mathcal{D}}$ be a $\mathbb{P}$-name for family of $<\lambda$-many dense subsets of $\mathbb{R}$. Then, using the $\kappa^+$-cc, we can find $\beta<\lambda$ such that both $\dot{\mathbb{R}}$ and $\dot{\mathcal{D}}$ belong to $V^{\mathbb{P}_\beta}$. We can choose then, $\gamma < \lambda$ so that $\mathbb{R}$ is the $\gamma$-th name in $V^{\mathbb{P}_\beta}$ for a poset with property $(\ast_\kappa)$. Hence, in the model $V^{\mathbb{P}_{\pi(\beta, \gamma)+1}}$, the generic  for $\mathbb{R}$ intersects all dense sets in $\mathcal{D}$.
 
The argument above is enough to obtain the full $\MA(\kappa\hbox{-good-Knaser})$ in $V^{\mathbb{P}}$:

\begin{clm}
If $\mathbb{R}$ is \emph{$\kappa$-good-Knaster} poset in $V^{\mathbb{P}}$ and $\mathcal{D}$ is a collection of $<\lambda$-many dense sets in $\mathbb{R}$, then there is $\mathbb{R}' \subseteq \mathbb{R}$ of cardinality $<\lambda$
which is also \emph{$\kappa$-good-Knaster} such that the sets in $\mathcal{D}$ are dense in $\mathbb{R}'$. 
\end{clm} 
\begin{proof}
Given a dense set $D \in \mathcal{D}$, there exists a maximal antichain $A_D \subseteq D$ and using the stationary $\kappa^+$-Knaster condition, this antichain has size at most $\kappa$. Consider then, the poset $S$ generated by the set of antichains $\{A_D: D \in \mathcal{D}\}$ and has size $<\lambda$ (because $\lambda^{<\kappa}=\lambda$). Now, consider the closure of $S$ under properties (2), (3) and (4) in Definition \ref{good} and notice that this process does not increase its size. Call the resulting poset $\mathbb{R}'$ and note that it has the desired size and it is an element of the class \emph{$\kappa-\hbox{good-Knaster}$}. Finally, if $H \subseteq \mathbb{R}'$ is a generic intersecting all the dense sets in $\mathcal{D}$, we can extend it to a filter $G \supseteq H$, $G \subseteq \mathbb{R}$ meeting all sets in $\mathcal{D}$.
\end{proof}
\end{proof}

There have been other attempts to get higher analogues of Martin's axiom at $\kappa=\aleph_1$. Specifically, let us mention one due to Baumgartner (see also \cite{ShSt:MAS, Sh:WMA}):

\begin{dfn}[Baumgartner's axiom \cite{Baum83}]
Let $\mathbb{P}$ be a partial order satisfying the following conditions:
\begin{itemize}
    \item $\mathbb{P}$ is countably closed.
    \item $\mathbb{P}$ is well-met.
    \item $\mathbb{P}$ is $\aleph_1$-linked.
\end{itemize}
Then if $\kappa< 2^{\aleph_1}$ and $\{\mathcal{D}_\alpha: \alpha<\kappa\}$ is a collection of dense sets of $\mathbb{P}$, then there exists a generic filter $G \subseteq \mathbb{P}$ intersecting all sets $\mathcal{D}_\alpha$. 
\end{dfn}

Baumgartner also proved that the former axiom is consistent with $2^{\aleph_0}=\aleph_1$ and $2^{\aleph_1}=\kappa$, where $\kappa \geq \aleph_1$ is regular.

\begin{comment}

\subsection{Questions about $\mf d(\kappa)$ and $\mf d_{\cl}(\kappa)$}

Cummings and Shelah remark that whether $\mf d(\kappa)<\mf d_{\cl}(\kappa)$  is possible is connected to some PCF problems, but that remark is never made really precise. It would be interesting to get an actual equivalence, which would probably involve understanding the relevant cases of the Revised GCH that is used to show $\mf d(\kappa)=\mf d_{\cl}(\kappa)$ when $\kappa\geq \beth_\oo$.
\medskip

I don't know if these questions lead to anywhere interesting.

\begin{quest}
Can we add a $\leq_{\cl}$-dominating  $f\in \omg^\omg$ which is not $\leq^*$-dominating?
\end{quest}

%\todo{this is some remark.}

Note that if a function $f\in \omg^\omg$ is $\leq_{\cl}$-dominating over $\omg^\omg\cap V$ then any $f$-closed club $D\subseteq \omg$ is not in $V$. In other words, adding $\leq_{\cl}$-dominating functions will necessarily introduce new clubs.

\begin{quest}
Can we add a $\leq_{\cl}$-dominating  $f\in \omg^\omg$ without diagonalizing the club filter?% Maybe no new reals, and with ccc forcing?
\end{quest}

Note that if we diagonalize the club filter then $\leq^*$-dominating functions are added.
\end{comment}

\printbibliography

\end{document}